\documentclass[graybox]{svmult}

\usepackage{mathptmx}       
\usepackage{helvet}         
\usepackage{courier}        
\usepackage{type1cm}        
\usepackage{makeidx}         
\usepackage{graphicx}        
\usepackage{multicol}        
\usepackage[bottom]{footmisc}

\usepackage{natbib}
\bibpunct{(}{)}{;}{d}{,}{,}

\usepackage{amsmath,amsfonts,mathrsfs}
\usepackage{epsfig,graphicx,color}
\usepackage{verbatim,float}
\usepackage{subfigure}

\makeindex             

\begin{document}

\title*{A McKean optimal transportation perspective on Feynman-Kac formulae 
with application to data assimilation}
\titlerunning{A McKean optimal transportation perspective on data assimilation}
\author{Yuan Cheng and Sebastian Reich}
\institute{Yuan Cheng \at  Institut f\"ur Mathematik, Universit\"at Potsdam, Am Neuen Palais
10, D-14469 Potsdam, Germany, \email{yuan.cheng@uni-potsdam.de} \and
Sebastian Reich \at Institut f\"ur Mathematik, Universit\"at Potsdam, Am Neuen Palais
10, D-14469 Potsdam, Germany, and Department of Mathematics and Statistics, University of Reading,
Whiteknights, PO Box 220, Reading, RG6 6AX, UK, \email{sreich@math.uni-potsdam.de}
}
%
%
\maketitle

\date{26/09/2013}

\abstract{Data assimilation is the task of combining mathematical
  models with observational data. From a mathematical perspective data
assimilation leads to Bayesian inference problems which can be formulated in
terms of Feynman-Kac formulae. In this paper we focus on the sequential
nature of many data assimilation problems and their numerical
implementation in form of Monte Carlo methods. We demonstrate how
sequential data assimilation can be interpreted as
time-dependent Markov processes, which is often referred to as the
McKean approach to Feynman-Kac formulae. It is shown that the McKean
approach has very natural links to coupling of random variables and
optimal transportation. This link allows one to propose novel
sequential Monte Carlo methods/particle filters. In combination with
localization these novel algorithms have the
potential of beating the curse of dimensionality, which has prevented
particle filters from being applied to spatially extended systems.}

%
%
%

\section{Introduction}
\label{sec:1}

This paper is concerned with Monte Carlo methods for approximating
expectation values for sequences of probability density functions
(PDFs) $\pi^n(z)$, $n\ge 0$, $z\in {\cal Z}$. We assume that these PDFs arise
sequentially from a Markov process with given transition kernel
$\pi(z|z')$ and are modified by weight functions $G^n(z)\ge0$ at each
iteration index $n\ge1$. More precisely, the PDFs satisfy the
recursion
\begin{equation} \label{sec1:FK}
\pi^{n}(z) = \frac{1}{C} G^n(z) \int_{\cal Z} \pi(z|z')
\pi^{n-1}(z'){\rm d}z'
\end{equation}
with the constant $C$ chosen such that $\int_{\cal Z} \pi^n(z){\rm d}z
= 1$.

A general mathematical framework for such problems is provided by the
Feynman-Kac formalism as discussed in detail in \cite{sr:DelMoral}.\footnote{The classic 
Feynman-Kac formulae provide a connection between stochastic processes and 
solutions to partial differential equations. Here
we use a generalization which links discrete-time stochastic processes to
sequences of marginal distributions and associated expectation values. In addition to 
sequential Bayesian inference, which primarily motivates this review article, 
applications of discrete-time Feynman-Kac formula of type (\ref{sec1:FK})
can, for example, be found in non-equilibrium molecular dynamics, where
the weight functions $G^n$ in (\ref{sec1:FK}) corresponds to the incremental work exerted on
a molecular system at time $t_n$. See \cite{sr:stoltz} for more details.}
In order to apply Monte Carlo methods to (\ref{sec1:FK}) it is useful
to reformulate (\ref{sec1:FK}) in terms of modified Markov
processes with transition kernel $\pi^n(z|z')$,  
which satisfy the consistency condition
\begin{equation} \label{sec1:comp}
\pi^{n}(z) = \int_{\cal Z} \pi^n(z|z') \pi^{n-1}(z'){\rm d}z'.
\end{equation}
This reformulation has been called the McKean approach to
Feynman-Kac models in  \cite{sr:DelMoral}.
\footnote{\cite{sr:McKean66} pioneered the study of stochastic processes
which are generated by stochastic differential equations for which the diffusion term depends
on the time-evolving marginal distributions $\pi(z,t)$. Here we utilize a generalization of this idea 
to discrete-time Markov processes which allows for transition kernels $\pi^n(z|z')$ to depend on the marginal distributions $\pi^n(z)$.} 
Once a particular McKean model is available, a Monte Carlo
implementation reduces to sequences of particles
$\{z_i^n \}_{i=1}^M$ being generated sequentially by
\begin{equation} \label{sec1:MC}
z_i^{n} \sim \pi^n(\cdot|z_i^{n-1}), \qquad i=1,\ldots,M,
\end{equation}
for $n = 0,1,\ldots,N$. In other words, $z^{n}_i$ is the realization of a random variable with
(conditional) PDF $\pi^n(z|z_i^{n-1})$. Such a Monte Carlo method
constitutes a particular instance of the far more general class of
sequential Monte Carlo methods (SMCMs) \citep{sr:Doucet}.

While there are many applications that naturally give rise to
Feynman-Kac formulae \citep{sr:DelMoral}, we will focus in this paper
on Markov processes for which the underlying transition kernel $\pi(z|z')$ is
determined by a deterministic dynamical system and that we wish to
estimate its current state $z^n$ from partial and noisy observations
$y^n_{\rm obs}$. The weight function $G^n(z)$ of a Feynman-Kac recursion (\ref{sec1:FK})
is in this case given by the likelihood of observing $y_{\rm obs}^n$
given $z^n$ and we encounter a particular application of 
\emph{Bayesian inference} \citep{sr:jazwinski,sr:stuart10a}.
The precise mathematical setting and the Feynman-Kac formula for the
associated data assimilation problem will be discussed in Section
\ref{sec:2}. Some of the standard Monte Carlo
approaches to Feynman-Kac formulae will be summarized in  Section \ref{sec:3}.

It is important to note that the consistency condition
(\ref{sec1:comp}) does not specify a McKean model uniquely. In other words,
given a Feynman-Kac recursion (\ref{sec1:FK}) there are many options to define an associated 
McKean model $\pi^n(z|z')$.  It has been suggested independently 
by \cite{sr:reich10,sr:cotterreich} and \cite{sr:marzouk11} in the
context of Bayesian inference that optimal transportation
\citep{sr:Villani} can be used to couple the prior and posterior
distributions. This idea generalizes to all Feynman-Kac formulae and leads
to optimal in the sense of optimal transportation McKean models. This
optimal transportation approach to McKean models will be developed in
detail in Section \ref{sec:4} of this paper. 

Optimal transportation problems lead to a nonlinear elliptic PDE,
called the Monge-Ampere equation \citep{sr:Villani}, which is very hard to tackle
numerically in space dimensions larger than one. On the other hand,
optimal transportation is an infinite-dimensional generalization 
\citep{sr:mccann95,sr:Villani2} of the classic linear transport
problem \citep{sr:strang}. This interpretation is very attractive in
terms of Monte Carlo methods and gives rise to a novel SMCM of type
(\ref{sec1:MC}), which we call the ensemble transform particle filter (ETPF)
\citep{sr:reich13}. The ETPF is based on a linear transformation of the forecast particles 
\begin{equation}\label{sec1:forecast}
z_i^f \sim \pi(\cdot|z_i^{n-1}), \qquad i = 1,\ldots,M,
\end{equation}
of type
\begin{equation} \label{sec1:trans}
z_j^{n} = \sum_{i=1}^M z_i^f s_{ij}
\end{equation}
with the entries $s_{ij}\ge 0$ of the transform matrix $S\in
\mathbb{R}^{M\times M}$ being determined by an appropriate linear
transport problem. Even more remarkably, it turns out that SMCMs which
resample in each iteration as well as the popular class of ensemble
Kalman filters (EnKFs) \citep{sr:evensen} also fit into the linear transform
framework of (\ref{sec1:trans}). We will discuss
particle/ensemble-based sequential data assimilation algorithms
within the unifying framework of linear ensemble transform filters in Section
\ref{sec:5}. This section also includes an extension to spatially extended
dynamical systems using the concept of localization \citep{sr:evensen}. 
Section \ref{sec:6} provides numerical results for the
Lorenz-63 \citep{sr:lorenz63} and the Lorenz-96 \citep{sr:lorenz96}
models. The results for the 40 dimensional Lorenz-96 indicate that the
ensemble transform particle filter
with localization can beat the curse of dimensionality which has so far
prevented SMCMs from being used for high-dimensional systems
\citep{sr:bengtsson08}. A brief historical account of data assimilation and
filtering is given in Section 7.

We mention that, while the focus of this paper is on state estimation for
deterministic dynamical systems, the results can easily be extended to
stochastic models as well as combined state and parameter estimation
problems. Furthermore, possible applications include
all areas in which SMCMs have successfully been used. We
refer, for example, to navigation,  computer vision, and cognitive
sciences (see, e.g., \cite{sr:Doucet,sr:mumford03} and references therein).

%
%

\section{Data assimilation and Feynman-Kac formula}
\label{sec:2}

Consider a deterministic \emph{dynamical system}\footnote{Even though this review article assumes
deterministic evolution equations, the results presented here can easily 
be generalized to evolution equations with stochastic model errors.}
\begin{equation} \label{sec2:DS}
z^{n+1} = \Psi (z^n)
\end{equation}
with state variable $z\in \mathbb{R}^{N_z}$, iteration index $n\ge 0$
and given initial $z^0 \in {\cal Z} \subset \mathbb{R}^{N_z}$. We assume
that $\Psi$ is a diffeomorphism on $\mathbb{R}^{N_z}$ and that 
$\Psi({\cal Z}) \subseteq {\cal Z}$, which implies that the
iterates $z^n$ stay in ${\cal Z}$ for all $n\ge 0$. Dynamical systems 
of the form (\ref{sec2:DS}) often arise as the time-$\Delta t$-flow maps
of differential equations
\begin{equation} \label{sec2:ODE}
\frac{{\rm d}z}{{\rm d}t} = f(z).
\end{equation}

In many practical applications the initial state $z^0$ is not
precisely known. We may then assume that our uncertainty about the
correct initial state can, for example, be quantified in terms of 
ratios of frequencies of occurrence.  More precisely, the 
\emph{ratio of frequency of occurrence} (RFO) of two initial conditions $z_a^0 \in {\cal Z}$ and
$z_b^0 \in {\cal Z}$ is defined as the ratio of frequencies of occurrence for
the two associated $\varepsilon$-neighborhoods ${\cal U}_\varepsilon
(z_a^0)$ and ${\cal U}_\varepsilon (z_b^0)$, respectively, and taking
the limit $\varepsilon \to 0$. Here ${\cal U}_\varepsilon (z)$ is defined as
\[
{\cal U}_\varepsilon (z) = \{z' \in \mathbb{R}^{N_z}: \|z'-z\|^2 \le \varepsilon\}
\]
It is important to note that the volume of both neighborhoods are identical,
\emph{i.e.}, $V({\cal U}_\varepsilon(z_a^0)) = V({\cal U}_\varepsilon (z_b^0))$.

From a frequentist perspective, RFOs can be thought of as arising from repeated experiments 
with one and the same dynamical system (\ref{sec2:DS}) under varying initial conditions and 
upon counting how often $z^0 \in {\cal U}_\varepsilon(z^0_a)$ relative to $z^0 \in {\cal
  U}_{\varepsilon}(z^0_b)$. There will, of course, be many instances for which
  $z^0$ is neither in ${\cal U}_\varepsilon(z_a^0)$ nor ${\cal U}_\varepsilon (z_b^0)$.
   Alternatively, one can take a Bayesian perspective and think of RFOs as our subjective 
belief about a $z_a^0$ to actually arise as the
initial condition in (\ref{sec2:DS}) relative to another initial condition $z_b^0$. 
The later interpretation is also applicable in case only a single experiment 
with (\ref{sec2:DS}) is conducted.

Independent of such a statistical interpretation of the RFO,
we assume the availability of a function $\tau (z)>0$ such that the RFO can be expressed as
\begin{equation} \label{sec2:fittomodel}
\mbox{RFO} = \frac{\tau(z_a^0)}{\tau(z_b^0)}
\end{equation}
for all pairs of initial conditions from ${\cal Z}$. 


Provided that
\[
\int_{\cal Z} \tau(z) {\rm d}z < \infty
\]
we can introduce the probability density function (PDF) 
\[
\pi_{Z^0}(z) = \frac{\tau(z)}{\int_{\cal Z} \tau(z) {\rm d}z}
\]
and interpret initial conditions $z^0$ as realizations of a random
variable $Z^0:\Omega \to \mathbb{R}^{N_z}$ with PDF
$\pi_{Z^0}$.\footnote{We have assumed the existence of an underlying
  probability space $(\Omega,{\cal F},\mathbb{P})$. The specific
  structure of this probability space does not play a role in the
  subsequent discussions.} We remark that most of our subsequent
discussions carry through even if $\int_{\cal
  Z}\tau(z) {\rm d}z$ is unbounded as long as the RFOs remain well-defined. 

So far we have discussed RFOs for initial conditions.
But one can also consider such ratios for any iteration index $n\ge 0$, \emph{i.e.},
for solutions
\[
z^n_a = \Psi^{n} (z^0_a)
\]
and
\[
z^n_b = \Psi^{n} (z^0_b).
\]
Here $\Psi^{n}$ denotes the $n$-fold application of
$\Psi$. The RFO at iteration index $n$ is now defined as the
ratio of frequencies of occurrence for the two associated $\varepsilon$-neighborhoods
${\cal U}_\varepsilon(z_a^n)$ and ${\cal U}_\varepsilon (z_b^n)$,
respectively, in the limit $\varepsilon \to 0$. We pull this ratio
back to $n=0$ and find that
\[
\mbox{RFO}(n) \approx \frac{ \tau(\Psi^{-n}(z^n_a))
  V_a}{\tau(\Psi^{-n}(z^n_b))V_b}
\]
for $\varepsilon$ sufficiently small,
where 
\[
V_{a/b} = V(\Psi^{-n}({\cal U}_\varepsilon (z_{a/b}^n))) :=
\int_{\Psi^{-n}({\cal U}_\varepsilon (z_{a/b}^n))} {\rm d}z
\]
denote the volumes of $\Psi^{-n}({\cal
  U}_\varepsilon(z^n_{a/b}))$ and $\Psi^{-n}$ refers to the inverse of
$\Psi^{n}$. These two volumes can be approximated as
\[
V_{a/b} \approx V({\cal U}_\varepsilon (z^0_{a/b}))
\times  |D\Psi^{-n}(z^n_{a/b})|
\]
for $\varepsilon>0$ sufficiently small.
Here $D\Psi^{-n}(z) \in \mathbb{R}^{N_z\times N_z}$ stands for the Jacobian matrix of
partial derivatives of $\Psi^{-n}$ at $z$ and $|D\Psi^{-n}(z)|$ for its determinant.
Hence, upon taking the limit $\varepsilon \to 0$, we obtain
\begin{align*}
\mbox{RFO}(n) &= \frac{\tau(\Psi^{-n}(z^n_a))
  |D\Psi^{-n}(z^n_a)|}{\tau(\Psi^{-n}(z^n_b))
  |D\Psi^{-n}(z^n_b)|} \\
&= \frac{\pi_{Z^0}(\Psi^{-n}(z^n_a))
  |D\Psi^{-n}(z^n_a)|}{\pi_{Z^0}(\Psi^{-n}(z^n_b))
  |D\Psi^{-n}(z^n_b)|}.
\end{align*}
Therefore we may interpret solutions $z^n$ for fixed iteration index
$n\ge 1$ as realizations of a random variable $Z^n:\Omega \to \mathbb{R}^{N_z}$ with PDF
\begin{equation} \label{sec2:marginal}
\pi_{Z^n}(z) = \pi_{Z^0}(\Psi^{-n}(z)) |D\Psi^{-n}(z)|.
\end{equation}
These PDFs can also be defined recursively using
\begin{equation} \label{sec2:Markov1}
\pi_{Z^{n+1}}(z) = \int_{\cal Z} \delta (z - \Psi(z'))\,\pi_{Z^n}(z')\,
{\rm d}z' .
\end{equation}
Here $\delta(\cdot)$ denotes the Dirac delta function, which satisfies
\[
\int_{\mathbb{R}^{N_z}} f(z) \delta (z-\bar z) {\rm d}z = f(\bar z)
\]
for all smooth functions $f:\mathbb{R}^{N_z} \to \mathbb{R}$. \footnote{The Dirac delta function
$\delta (z-\bar z)$ provides  a convenient short-hand for the point measure $\mu_{\bar z}({\rm d}z)$.}
In other words, the dynamical system (\ref{sec2:DS}) induces a \emph{Markov
process}, which we can also write as
\[
Z^{n+1} = \Psi(Z^n)
\]
in terms of the random variables $Z^n$, $n\ge 0$.

The sequence of random variables $\{Z^n\}_{n=0}^N$ for fixed $N\ge 1$ gives rise to the
finite-time stochastic process $Z^{0:N}:\Omega \to {\cal Z}^{N+1}$ with realizations
\[
z^{0:N} := (z^0,z^1,\ldots,z^N) = Z^{0:N}(\omega), \quad \omega \in \Omega,
\]
that satisfy (\ref{sec2:DS}). The joint distribution of
$Z^{0:N}$, denoted by $\pi_{Z^{0:N}}$, is formally\footnote{To be
  mathematically precise one should talk about the joint measure 
\[
\mu_{Z^{0:N}}({\rm d}z^0,\dots,{\rm d}z^N) = 
\mu_{Z^0}({\rm d}z^0)  \mu_{\Psi(z^0)}({\rm d}z^1) \cdots
\mu_{\Psi(z^{N-1})}({\rm d}z^N)
\]
with initial measure $\mu_{Z^0}({\rm d}z^0) = \pi_{Z^0}(z^0){\rm d}z^0$.}
given by
\begin{equation} \label{sec2:joint}
\pi_{Z^{0:N}}(z^0,\ldots,z^N)  = 
\pi_{Z^0}(z^0)\,\delta (z^1-\Psi(z^0)) \cdots
\delta (z^N -\Psi(z^{N-1}))
\end{equation}
and (\ref{sec2:marginal}) is the marginal of $\pi_{Z^{0:N}}$ in $z^n$, $n=1,\ldots,N$. 

Let us now consider the situation where (\ref{sec2:DS})
serves as a model for an unknown physical process with realization
\begin{equation} \label{sec2:truth}
z^{0:N}_{\rm ref} = (z^0_{\rm ref},z^1_{\rm ref},\ldots,z^N_{\rm ref}) .
\end{equation}
In the classic filtering/smoothing setting
\citep{sr:jazwinski,sr:crisan} one assumes that there 
exists an $\omega_{\rm ref} \in \Omega$ such that 
\[
z^{0:N}_{\rm ref} = Z^{0:N}(\omega_{\rm ref}).
\]
In practice such an assumption is highly unrealistic and 
the reference trajectory (\ref{sec2:truth}) may instead follow an iteration
\begin{equation} \label{sec2:DST}
z^{n+1}_{\rm ref} = \Psi_{\rm ref}(z^n_{\rm ref})
\end{equation}
with unknown initial $z^0_{\rm ref}$ and unknown $\Psi_{\rm ref}$. Of
course, it should hold that $\Psi$ in (\ref{sec2:DS}) is close to 
$\Psi_{\rm ref}$ in an appropriate mathematical sense.

Independently of such assumptions, we assume that $z^{0:N}_{\rm ref}$ is accessible to us through
partial and noisy observations of the form
\[
y^n_{\rm obs} = h(z_{\rm ref}^n) + \xi^n, \quad n=1,\ldots,N,
\]
where $h:{\cal Z}\to\mathbb{R}^{N_y}$ is called the \emph{forward} or {observation map} 
and the $\xi^n$'s are realizations of independent and identically
distributed Gaussian random variables with mean zero and 
covariance matrix $R \in \mathbb{R}^{N_y \times N_y}$. Estimating
$z_{\rm ref}^n$ from $y^n_{\rm obs}$ constitutes a classic inverse
problem \citep{sr:Tarantola}.

The \emph{ratio of fits to data} (RFD) of two realizations $z_a^{0:N}$ and $z_b^{0:N}$ from 
the stochastic process $Z^{0:N}$ is defined as
\[
\mbox{RFD} = \frac{\prod_{n=1}^N
  e^{-\frac{1}{2}(h(z^n_a)-y^n_{\rm obs})^T R^{-1}
      (h(z^n_a)-y_{\rm obs}^n) }}{\prod_{n=1}^N
e^{-\frac{1}{2}(h(z^n_b)-y^n_{\rm obs})^T R^{-1}
      (h(z^n_b)-y_{\rm obs}^n) }} .
\]
Finally we define the \emph{ratio of fits to model and data} (RFMD) 
of a $z_a^{0:N}$ versus a $z_b^{0:N}$ given the model and the
observations as follows:
\begin{align} \nonumber
\mbox{RFMD} &= \mbox{RFD} \times \mbox{RFO}(0)\\  \nonumber
&= \frac{\prod_{n=1}^N
  e^{-\frac{1}{2} (h(z^n_a)-y^n_{\rm obs})^T R^{-1}
      (h(z^n_a)-y_{\rm obs}^n) }}{\prod_{n=1}^N
e^{-\frac{1}{2} (h(z^n_b)-y^n_{\rm obs})^T R^{-1}
      (h(z^n_b)-y_{\rm obs}^n) }} 
\frac{\pi_{Z^0}(z_a^0)}{\pi_{Z^0}(z_b^0)}
 \\ \label{sec2:fitmd}
&=\frac{\prod_{n=1}^N
  e^{-\frac{1}{2}(h(z^n_a)-y^n_{\rm obs})^T R^{-1}
      (h(z^n_a)-y_{\rm obs}^n) }}{\prod_{n=1}^N
e^{-\frac{1}{2}(h(z^n_b)-y^n_{\rm obs})^T R^{-1}
      (h(z^n_b)-y_{\rm obs}^n) }} 
\frac{\pi_{Z^{0:N}}(z_a^{0:N})}{\pi_{Z^{0:N}}(z_b^{0:N})} .
\end{align}
The simple product structure arises since the uncertainty in the
initial conditions is assumed to be independent of the measurement errors
and the last identity follows from the fact that our dynamical model is deterministic. 

Again we may translate this combined ratio into a PDF
\begin{equation} \label{sec2:feynman}
\pi_{Z^{0:N}}(z^{0:N}|y_{\rm obs}^{1:N}) = 
\frac{1}{C} \prod_{n=1}^N
  e^{-\frac{1}{2} (h(z^n)-y^n_{\rm obs})^T R^{-1}
      (h(z^n)-y_{\rm obs}^n) } \,\pi_{Z^{0}}(z^0),
\end{equation}
where $C>0$ is a normalization constant depending only on $y_{\rm obs}^{1:N}$. 
This PDF gives the probability distribution in $z^{0:N}$ conditioned on the given set
of observations 
\[
y_{\rm obs}^{1:N} = (y^1_{\rm obs},\ldots,y_{\rm obs}^N).
\]
The PDF (\ref{sec2:feynman}) is, of course, also conditioned on (\ref{sec2:DS})
and the initial PDF $\pi_{Z^0}$. This dependence is not explicitly
taken account of in order to avoid additional notational clutter. 

The formulation (\ref{sec2:feynman}) is an instance of Bayesian
inference on the one hand, and an instance of the Feynman-Kac
formalism on the other. Within the Bayesian perspective,
$\pi_{Z^{0}}$ (or, equivalently, $\pi_{Z^{0:N}}$) represents the prior distribution, 
\[
\pi_{Y^{1:N}}({\rm y}^{1:n}|z^{0:n}) = \frac{1}{(2\pi)^{N_y N/2} |R|^{N/2}}
 \prod_{n=1}^N
  e^{-\frac{1}{2} (h(z^n)-y^n)^T R^{-1}
      (h(z^n)-y^n) } 
\]
the compounded likelihood function, and (\ref{sec2:feynman}) the posterior PDF given
an actually observed $y^{1:n} = y^{1:n}_{\rm obs}$. The Feynman-Kac formalism
is more general and includes a wide range of applications for which
an underlying stochastic process is modified by weights $G^n(z^n)\ge 0$.
These weights then replace the likelihood functions
\[
\pi_{Y}(y^n_{\rm obs}|z^n) = \frac{1}{(2\pi)^{N_y/2} |R|^{1/2}}  e^{-\frac{1}{2} (h(z^n)-
y^n_{\rm obs})^T R^{-1} (h(z^n)-y^n_{\rm obs}) } 
\]
in (\ref{sec2:feynman}). The functions $G^n:\mathbb{R}^{N_z} \to
\mathbb{R}$ can depend on the iteration index, as in
\[
G^n(z) := \pi_{Y}(y_{\rm obs}^n|z)
\]
or may be independent of the iteration index. See
\cite{sr:DelMoral} for further details on the Feynman-Kac formalism
and \cite{sr:stoltz} for a specific (non-Bayesian) application in the context of 
non-equilibrium molecular dynamics.

Formula (\ref{sec2:feynman}) is hardly ever explicitly accessible and
one needs to resort to numerical approximations whenever one wishes
to either compute the expectation value
\[
\mathbb{E}_{Z^{0:N}}[f(z^{0:N})|y_{\rm obs}^{1:N}] = \int_{{\cal Z}^{N+1}}
f(z^{0:N}) \,\pi_{Z^{0:N}}(z^{0:N}|y_{\rm obs}^{1:N}) {\rm d}z^0 \cdots {\rm d}z^N
\]
of a given function $f:{\cal Z}^{N+1} \to \mathbb{R}$ or the 
\[
\mbox{RFMD} = 
\frac{\pi_{Z^{0:N}}(z_a^{0:N}|y_{\rm
    obs}^{1:N})}{\pi_{Z^{0:N}}(z_b^{0:N}|y_{\rm obs}^{1:N})}
\]
for given trajectories $z^{0:N}_a$ and $z^{0:N}_b$. Basic Monte Carlo
approximation methods will be discussed in Section \ref{sec:3}. Alternatively, one may seek the
\emph{maximum a posteriori (MAP) estimator} $z^0_{\rm MAP}$, which is defined as the
initial condition $z^0$ that maximizes (\ref{sec2:feynman}) or, formulated alternatively,
\[
z^0_{\rm MAP} = \mbox{arg} \,\inf L(z^0) ,\qquad L(z^0) := -\log \pi_{Z^{0:N}}(z^{0:N}|y_{\rm obs}^{1:N}) 
\]
\citep{sr:kaipio,sr:Lewis,sr:Tarantola}. The MAP estimator is closely related to \emph{variational 
data assimilation techniques}, such as 3D-Var and 4D-Var, widely used in meteorology \citep{sr:daley,sr:kalnay}. 

In many applications expectation values need to be computed for functions $f$ which
depend only on a single $z^n$. Those expectation values can be obtained by first
integrating out all components in (\ref{sec2:feynman}) except for $z^n$. We denote the
resulting marginal PDF by $\pi_{Z^n}(z^n|y_{\rm obs}^{1:N})$. The case $n=0$ plays 
a particular role since
\[
\mbox{RFMD} = 
\frac{\pi_{Z^{0}}(z_a^{0}|y_{\rm
    obs}^{1:N})}{\pi_{Z^{0}}(z_b^{0}|y_{\rm obs}^{1:N})}
\]
and
\[
\pi_{Z^n}(z^n|y_{\rm obs}^{1:N}) = \pi_{Z^0}(\Psi^{-n}(z^n)|y_{\rm obs}^{1:N})
|D\Psi^{-n}(z^n)|
\]
for $n=1,\ldots,N$.
These identities hold because our dynamical system (\ref{sec2:DS}) is deterministic and invertible. 
In Section \ref{sec:4} we will discuss recursive approaches for
determining the marginal $\pi_{Z^N}(z^N|y_{\rm obs}^{1:N})$ in terms of Markov
processes. Computational techniques for implementing such recursions will be discussed 
in Section \ref{sec:5}.


%
%
%
%

\section{Monte Carlo methods in path space}
\label{sec:3}

In this section, we briefly summarize two popular Monte Carlo
strategies for computing expectation values with respect to the 
complete conditional distribution $\pi_{Z^{0:N}}(\cdot|y_{\rm
  obs}^{1:N})$. We start with the classic importance sampling Monte
Carlo method.

\subsection{Ensemble prediction and importance sampling}

Ensemble prediction is a Monte Carlo method for assessing the marginal PDFs
(\ref{sec2:marginal}) for $n=0,\ldots,N$. One first generates $z_i^0$, $i=1,\ldots,M$,
independent samples from the initial PDF $\pi_{Z^0}$. Here samples are generated
such that the probability of being in ${\cal U}_\varepsilon(z)$ is 
\[
\int_{{{\cal U}_\varepsilon}(z)} \pi_{Z^0}(z') {\rm d}z' \approx
V({\cal U}_\varepsilon (z)) \times \pi_{Z^0}(z).
\]
Since samples $z_i^0$ are generated such that their RFOs are all equal, 
it follows that their ratio of fits to model is identical to one. Furthermore, the expectation value
of a function $f$ with respect to $Z^0$ is approximated by the
familiar empirical estimator
\[
\bar f_M := \frac{1}{M} \sum_{i=1}^M f(z_i^0).
\]

The initial ensemble $\{z_i^0\}_{i=1}^M$ is propagated independently under
the dynamical system (\ref{sec2:DS}) for $n=0,\ldots,N-1$. This yields $M$ trajectories
\[
z_i^{0:N} = (z_i^0,z_i^1,\ldots,z_i^N),
\]
which provide independent samples from the $\pi_{Z^{0:N}}$ distribution.  With each of these samples we
associate the weight
\[
w_i = \frac{1}{C} 
\prod_{n=1}^N e^{-\frac{1}{2}  (h(z_i^n)-y_{\rm obs}^n)^T R^{-1} (h(z_i^n)-y_{\rm obs}^n) }
\]
with the constant of proportionality chosen such that $\sum_{i=1}^M w_i = 1$.

The ratio of fits to model and data for any pair of samples 
$z_i^{1:N}$ and $z_j^{1:N}$ from $\pi_{Z^{0:N}}$ is now simply given by $w_i/w_j$
and the expectation value of a function $f$ with respect to
$\pi_{Z^{0:N}}(z^{0:N}|y_{\rm obs}^{1:N})$ 
can be approximated by the empirical estimator
\[
\bar f = \sum_{i=1}^M w_i f(z^{0:N}_i) .
\]
This estimator is an instance of \emph{importance sampling} since samples from a distribution 
different from the target distribution, here $\pi_{Z^{0:N}}$,  are used to approximate the
statistics of the target distribution, here $\pi_{Z^{0:N}}(\cdot|y_{\rm
    obs}^{1:N})$. See \cite{sr:liu} and \cite{sr:Robert2004} for further details.

\subsection{Markov chain Monte Carlo (MCMC) methods}

Importance sampling becomes inefficient whenever the \emph{effective sample size}
\begin{equation} \label{sec3:SS}
M_{\rm eff} = \frac{1}{\sum_{i=1}^M w_i^2} \in [1,M]
\end{equation}
becomes much smaller than the sample size $M$. Under those circumstances it
can be preferable to generate dependent samples $z_i^{0:N}$ using 
MCMC methods. MCMC methods rely on a proposal step and a 
Metropolis-Hastings acceptance criterion. Note that only $z_i^0$ needs to be stored since the whole trajectory
is then uniquely determined by $z^n_i = \Psi^n(z_i^0)$. Consider, for
simplicity, the reversible proposal step
\[
z^0_p = z_i^0 + \xi,
\]
where $\xi$ is a realization of a random variable $\Xi$ with PDF $\pi_{\Xi}$ satisfying 
$\pi_{\Xi}(\xi) = \pi_{\Xi}(-\xi)$ and $z^0_i$ denotes the last accepted sample. The associated 
trajectory $z_p^{0:N}$ is computed using (\ref{sec2:DS}). Next one evaluates
the ratio of fits to model and data (\ref{sec2:fitmd}) for $z_a^{0:N} = z_p^{0:N}$ 
versus $z_b^{0:N} = z_i^{0:N}$, \emph{i.e.},
\[
\alpha := \frac{\prod_{n=1}^N
  e^{-\frac{1}{2} (h(z^n_p)-y^n_{\rm obs})^T R^{-1}
      (h(z^n_p)-y_{\rm obs}^n) }}{\prod_{n=1}^N
e^{-\frac{1}{2} (h(z^n_i)-y^n_{\rm obs})^T R^{-1}
      (h(z^n_i)-y_{\rm obs}^n) }} 
\frac{\pi_{Z^0}(z_p^0)}{\pi_{Z^0}(z_i^0)} .
\]
If $\alpha\ge 1$, then the proposal is accepted and the new sample for the initial condition 
is $z^0_{i+1} = z^0_p$. Otherwise the proposal is accepted with probability $\alpha$ and 
rejected with probability $1-\alpha$. In case of rejection one sets $z_{i+1}^0 = z_i^0$. 
Note that the accepted samples follow the $\pi_{Z^0}(z^0|y_{\rm obs}^{1:N})$ distribution
and not the initial PDF $\pi_{Z^0}(z^0)$.

A potential problem with MCMC methods lies in low acceptance rates and/or highly dependent
samples. In particular, if the distribution in $\pi_\Xi$ is too narrow then exploration of
phase space can be slow while a too wide distribution can potentially lead to low acceptance
rates. Hence one should compare the effective sample size (\ref{sec3:SS}) from an importance
sampling approach to the effective sample size of an MCMC implementation, which is 
inversely proportional to the integrated autocorrelation time of the samples. 
See \cite{sr:liu} and \cite{sr:Robert2004} for more details.

We close this section by referring to \cite{sr:sarkka} for further algorithms
for filtering and smoothing.  

%
%
%

\section{McKean optimal transportation approach}
\label{sec:4}

We now restrict the discussion of the Feynman-Kac formalism to the 
marginal PDFs $\pi_{Z^n}(z^n|y_{\rm obs}^{1:N})$. We have already seen in
Section \ref{sec:2} that the marginal PDF with $n=N$ plays a particularly important
role. We show that this marginal PDF can be recursively defined starting from
the PDF $\pi_{Z^0}$ for the initial condition $z^0$ of (\ref{sec2:DS}). For that reason
we introduce the forecast and analysis PDF at iteration index $n$ and denote them
by $\pi_{Z^n}(z^n|y_{\rm obs}^{1:n-1})$ and $\pi_{Z^n}(z^n|y_{\rm obs}^{1:n})$, 
respectively. Those PDFs are defined recursively by
\begin{equation} \label{sec4:Step1}
\pi_{Z^n}(z^n|y_{\rm obs}^{1:n-1}) = \pi_{Z^{n-1}}(\Psi^{-1}(z^n)|y_{\rm obs}^{1:n-1})|
D\Psi^{-1}(z^n)|
\end{equation}
and
\begin{equation} \label{sec4:Step2}
\pi_{Z^n}(z^n|y_{\rm obs}^{1:n}) = \frac{\pi_Y(y_{\rm obs}^n|z^n) 
\pi_{Z^n}(z^n|y_{\rm obs}^{1:n-1}) }{\int_{\cal Z} 
\pi_Y(y_{\rm obs}^n|z) 
\pi_{Z^n}(z|y_{\rm obs}^{1:n-1}) {\rm d}z}.
\end{equation}
Here (\ref{sec4:Step1}) simply propagates the analysis from index $n-1$ to
the forecast at index $n$ under the action of the dynamical system (\ref{sec2:DS}). 
Bayes' formula is then applied in (\ref{sec4:Step2}) in order to transform the forecast into the analysis
at index $n$ by assimilating the observed $y_{\rm obs}^n$. 

\begin{theorem}
Consider the sequence of forecast and analysis PDFs defined by the recursion 
(\ref{sec4:Step1})-(\ref{sec4:Step2}) for $n=1,\ldots,N$ 
with $\pi_{Z^0}(z^0)$ given. Then the analysis PDF at $n=N$ is equal to 
the Feynman-Kac PDF (\ref{sec2:feynman}) marginalized down to the single variable 
$z^N$.
\end{theorem}

\begin{proof}
We prove the theorem by induction in $N$. We first verify the claim for $N=1$. Indeed, by definition,
\begin{align*}
\pi_{Z^{1}}(z^1|y_{\rm obs}^1) &= \frac{1}{C_1} 
\int_{\cal Z} \pi_{Y}(y_{\rm obs}^1|z^1) \delta(z^1-\Psi(z^0)) \pi_{Z^0}(z^0) {\rm d}z^0 \\
&= \frac{1}{C_1} \pi_{Y}(y_{\rm obs}^1|z^1) \pi_{Z^0}(\Psi^{[-1]}(z^1))|D\Psi^{-1}(z^1)|\\
&= \frac{1}{C_1} \pi_{Y}(y_{\rm obs}^1|z^1) \pi_{Z^1}(z^1)
\end{align*}
and $\pi_{Z^1}(z^1)$ is the forecast PDF at index $n=1$. Here $C_1$ denotes the constant
of proportionality which only depends on $y_{\rm obs}^1$. 

The induction step from $N$ to $N+1$ follows from
the following line of arguments. We know by the induction assumption that the marginal
PDF at $N+1$ can be written as
\begin{align*}
\pi_{Z^{N+1}}(z^{N+1}|y_{\rm obs}^{1:N+1}) &= \int_{\cal Z}
\pi_{Z^{N:N+1}}(z^{N:N+1}|y_{\rm obs}^{1:N+1}) {\rm d}z^N\\
&= \frac{1}{C_{N+1}} \pi_Y(y_{\rm obs}^{N+1}|z^{N+1}) \int_{\cal Z} \delta (z^{N+1}-\Psi(z^N))
\pi_{Z^N}(z^N|y_{\rm obs}^{1:N}) {\rm d}z^N\\
&= \frac{1}{C_{N+1}} \pi_Y(y_{\rm obs}^{N+1}|z^{N+1}) \pi_{Z^{N+1}}(z^{N+1}|y_{\rm obs}^{1:N})
\end{align*}
in agreement with (\ref{sec4:Step2}) for $n=N+1$. Here
$C_{N+1}$ denotes the constant of proportionality that depends only on
$y_{\rm obs}^{N+1}$ and we have made use of the fact that the forecast PDF at index $n = N+1$
is, according to (\ref{sec4:Step1}), defined by 
\[
\pi_{Z^{N+1}}(z^{N+1}|y_{\rm obs}^{1:N}) = \pi_{Z^{N}}(\Psi^{-1} (z^{N+1})|y_{\rm obs}^{1:N})|
D\Psi^{-1} (z^{N+1})|. 
\] 
\hfill $\qed$
\end{proof}

While the forecast step (\ref{sec4:Step1}) is in the form of a Markov
process with transition kernel
\[
\pi_{\rm model}(z|z') = \delta(z-\Psi(z')),
\]
this does not hold for the analysis step (\ref{sec4:Step2}).  The
\emph{McKean approach} to (\ref{sec4:Step1})-(\ref{sec4:Step2}) is based on introducing
Markov transition kernels $\pi_{\rm data}^n(z|z')$, $n=1,\ldots,N$, 
for the analysis step (\ref{sec4:Step2}). In other words, the transition kernel $\pi_{\rm data}^n$ 
at iteration index $n$ has to satisfy the consistency condition
\begin{equation} \label{sec4:consistency} 
\pi_{Z^n}(z|y_{\rm obs}^{1:n}) = 
\int_{\cal Z} \pi_{\rm data}^n(z|z') \pi_{Z^n}(z'|y_{\rm obs}^{1:n-1})
{\rm d}z' .
\end{equation}
These transition kernels are not unique. The following kernel
\begin{equation} \label{sec4:DelMoral}
\pi_{\rm data}^n(z|z')  = \epsilon \pi_{Y}(y^n_{\rm obs}|z')
\delta(z-z') + \left(1-\epsilon \pi_{Y}(y^n_{\rm obs}|z')
\right) \pi_{Z^n}(z|y_{\rm obs}^{1:n})
\end{equation}
has, for example, been considered in \cite{sr:DelMoral}. Here $\epsilon \ge 0$ has
to be chosen such that 
\[
1-\epsilon \pi_{Y}(y^n_{\rm obs}|z) \ge 0
\]
for all $z\in \mathbb{R}^{N_z}$. Indeed, we find that
\begin{align*}
\int_{\cal Z} \pi_{\rm data}(z|z') \pi_{Z^n}(z'|y_{\rm obs}^{1:n-1}){\rm d}z' &=
\pi_{Z^n}(z|y_{\rm obs}^{1:n}) + \epsilon 
\pi_Y(y^n_{\rm obs}|z)\pi_{Z^n}(z|y_{\rm obs}^{1:n-1}) \,\,- \\
& \qquad \qquad \qquad  \epsilon \pi_{Z^n}(z|y^{1:n}_{\rm obs}) 
\pi_Y(y^n_{\rm obs}|y_{\rm obs}^{1:n-1}) \\
& = \pi_{Z^n}(z|y_{\rm obs}^{1:n}).
\end{align*}
The intuitive interpretation of this transition kernel is that one stays at $z'$ with probability
$p= \epsilon \pi_{Y}(y^n_{\rm obs}|z')$ while with probability
$(1-p)$ one samples from the analysis PDF $\pi_{Z^n}(z|y_{\rm obs}^{1:n})$.

Let us define the combined McKean-Markov transition kernel 
\begin{align} \nonumber
\pi^n(z^n|z^{n-1}) &:= \int_{\cal Z} 
\pi_{\rm data}^n(z^n|z)
\pi_{\rm model}(z|z^{n-1}) {\rm d}z \\ \nonumber 
&=  \int_{\cal Z} \pi_{\rm data}^n(z^n|z)
\delta  (z-\Psi(z^{n-1})) {\rm d}z \\
&= \pi_{\rm data}^n(z^n|\Psi(z^{n-1})) \label{sec3:kernel}
\end{align}
for the propagation of the analysis PDF from iteration index $n-1$ to
$n$. The combined McKean-Markov transition kernels $\pi^n$, 
$n=1,\ldots,N$, define a finite-time stochastic process $\hat Z^{0:N}
= \{\hat Z^n\}_{n=0}^N$ with $\hat Z^0 = Z^0$. The marginal PDFs
satisfy
\[
\pi_{{\hat Z}^n}(z^n|y_{\rm obs}^{1:n}) = 
\int_{\cal Z} \pi^n(z^n|z^{n-1}) \pi_{{\hat Z}^{n-1}}
(z^{n-1}|y_{\rm obs}^{1:n-1}) {\rm d}z^{n-1}.
\]

\begin{corollary}
The final time marginal distribution $\pi_{Z^N}(z^N|y_{\rm obs}^{1:N})$ of the Feynman-Kac
formulation (\ref{sec2:feynman}) is identical to the final time marginal distribution 
$\pi_{\hat Z^N}(z^N|y_{\rm obs}^{1:N})$ of the finite-time stochastic process $\hat Z^{0:N}$ 
induced by the McKean-Markov transition kernels $\pi^n(z^n|z^{n-1})$.
\end{corollary}

We summarize our discussion on the McKean approach 
in terms of analysis and forecast random variables, which constitute 
the basic building blocks for most current sequential data assimilation methods. 

\begin{definition}
Given a dynamic iteration (\ref{sec2:DS}) with PDF $\pi_{Z^0}$ for the initial conditions and 
observations $y_{\rm obs}^n$, $n=1,\ldots,N$, the McKean approach leads to a recursive 
definition of forecast $Z^{n,f}$ and analysis
$Z^{n,a}$ random variables. The iteration is started by declaring $Z^0$ the analysis $Z^{0,a}$ at 
$n=0$. The forecast at iteration index $n>0$ is defined by propagating the analysis at $n-1$ 
forward under the dynamic model (\ref{sec2:DS}), \emph{i.e.},
\begin{equation} \label{sec4:McKean1}
Z^{n,f} = \Psi(Z^{n-1,a}).
\end{equation}
The analysis $Z^{n,a}$ at iteration index $n$ is defined by applying the 
transition kernel $\pi_{\rm data}^{n}(z|z')$
to $Z^{n,f}$. In particular, if $z^{n,f} = Z^{n,f}(\omega)$ is a realized forecast at 
iteration index $n$, then the analysis is distributed according
to
\begin{equation} \label{sec4:McKean2}
Z^{n,a}|z^{n,f} \sim \pi_{\rm data}^{n}(z|z^{n,f}).
\end{equation}
If the marginal PDFs of $Z^{n,f}$ and $Z^{n,a}$ are denoted by $\pi_{Z^{n,f}}(z)$
and $\pi_{Z^{n,a}}(z)$, respectively, then the transition kernel
$\pi_{\rm data}^n$ has to satisfy the compatibility condition
(\ref{sec4:consistency}), \emph{i.e.},
\begin{equation} \label{sec4:compatible}
\int_{\cal Z} \pi_{\rm data}^n (z|z') \pi_{Z^{n,f}}(z'){\rm d}z' = 
\pi_{Z^{n,a}}(z)
\end{equation}
with
\[
\pi_{Z^{n,a}}(z) = \frac{\pi_Y(y_{\rm obs}^n|z) \pi_{Z^{n,f}}(z)}{\int_{\cal Z}
\pi_Y(y_{\rm obs}^n|z) \pi_{Z^{n,f}}(z){\rm d}z}.
\]
\end{definition}

The data related transition step (\ref{sec4:McKean2}) introduces randomness into the 
analysis of a given forecast value. This appears counterintuitive and, indeed, the main
purpose of the rest of this section  is to demonstrate that the transition kernel 
$\pi^n_{\rm data}$ can be chosen such that
\begin{equation} \label{sec4:McKean2b}
Z^{n,a} = \nabla_z \phi^{n}(Z^{n,f}),
\end{equation}
where $\phi^n:\mathbb{R}^{N_z} \to \mathbb{R}$ is an appropriate convex potential. In other words, 
the data-driven McKean update step can be reduced to a (deterministic) map
and the stochastic process $\hat Z^{0:N}$ is induced by the
deterministic recursion (or dynamical system)
\[
\hat Z^n = \nabla_z \phi^n(\Psi( \hat Z^{n-1}))
\]
with $\hat Z^0 = Z^0$. 

The compatibility condition (\ref{sec4:compatible}) with $\pi_{\rm data}^n(z|z')
= \delta (z-\nabla_z \psi^n(z'))$ reduces to
\begin{equation} \label{sec4:Monge}
\pi_{Z^{n,a}}(\nabla_z \psi^n(z)) |D\nabla_z \phi^n(z)| = 
\pi_{Z^{n,f}}(z),
\end{equation}
which constitutes a highly non-linear elliptic PDE for the potential
$\phi^n$.  In the remainder of this section we discuss under which
conditions this PDE has a solution. This discussion will also guide us
towards a numerical approximation technique that circumvents the 
need for directly solving (\ref{sec4:Monge}) either analytically or numerically.

Consider the forecast PDF $\pi_{Z^{n,f}}$ and the analysis PDF
  $\pi_{Z^{n,a}}$ at iteration index $n$. For simplicity of notion we
  drop the iteration index and simply write
$\pi_{Z^f}$ and $\pi_{Z^a}$, respectively. 

\begin{definition} A \emph{coupling} of
$\pi_{Z^f}$ and $\pi_{Z^a}$ consists of a pair
  $Z^{f:a} = (Z^f,Z^a)$ of random variables such that $Z^f \sim
  \pi_{Z^f}$, $Z^a \sim \pi_{Z^a}$, and $Z^{f:a} \sim
  \pi_{Z^{f:a}}$. The joint PDF $\pi_{Z^{f:a}}(z^f,z^a)$ on the product space
  $\mathbb{R}^{N_z}\times
\mathbb{R}^{N_z}$, is called the {\it transference plan} for
  this coupling. The set of all transference plans is denoted by
  $\Pi(\pi_{Z^f},\pi_{Z^a})$.\footnote{Couplings should be properly
    defined in terms of probability measures. A coupling between two
measures $\mu_{Z^f}({\rm d}z^f)$ and $\mu_{Z^a}({\rm d}z^a)$ consists
of a pair of random variables with joint measure $\mu_{Z^{f:a}}({\rm
  d}z^f,{\rm d}z^a)$ such that $\mu_{Z^f}({\rm d}z^f) = \int_{\cal Z} 
\mu_{Z^{f:a}}({\rm d}z^f,{\rm d}z^a)$ and $\mu_{Z^a}({\rm d}z^a) = \int_{\cal
  Z} \mu_{Z^{f:a}}({\rm d}z^f,{\rm d}z^a)$, respectively.} 
\end{definition}

Clearly, couplings always exist since one can use the trivial product coupling
\[
\pi_{Z^{f:a}}(z^f,z^a) = \pi_{Z^f}(z^f) \pi_{Z^a}(z^a),
\]
in which case the associated random variables $Z^f$ and $Z^a$ are
independent and each random variable follows its given marginal
distribution. Once a coupling has been found, a McKean transition
kernel is determined by
\[
\pi_{\rm data}(z|z') = \frac{\pi_{Z^{f:a}}(z',z)}{\int_{\cal Z}
  \pi_{Z^{f:a}}(z',z''){\rm d}z''}  .
\]
Reversely, any transition kernel $\pi_{\rm data}(z|z')$,
such as (\ref{sec4:DelMoral}), also induces a coupling. 

 A diffeomorphism $T: {\cal Z} \to {\cal Z}$ is called a \emph{transport map} if the induced random
  variable $Z^a= T(Z^f)$ satisfies
\[
\int_{{\cal Z}} f(z^a) \pi_{Z^a}(z^a){\rm d}z^a = \int_{{\cal Z}} f(T(z^f))
\pi_{Z^f}(z^f){\rm d}z^f 
\]
for all suitable functions $f:{\cal Z} \to \mathbb{R}$. The associated coupling 
\[
\pi_{Z^{f:a}}(z^f,z^a) = \delta (z^a - T(z^f))
  \pi_{Z^f}(z^f)
\]
is called a \emph{deterministic coupling}. Indeed, one finds that
\[
\int_{\cal Z} \pi_{Z^{f:a}}(z^f,z^a){\rm d}z^a = \pi_{Z^f}(z^f)
\]
and
 \[
\pi_{Z^a}(z^a) = \int_{\cal Z} \pi_{Z^{f:a}}(z^f,z^a){\rm d}z^f = \pi_{Z^f}(T^{-1}(z^a)) |DT^{-1}(z^a)|,
\]
respectively.

When it comes to actually choosing a particular coupling from the set
$\Pi(\pi_{Z^f},\pi_{Z^a})$ of all admissible ones, it appears preferable to
pick the one that maximizes the covariance or correlation 
between $Z^f$ and $Z^a$. But maximizing their covariance for given marginals has an important
geometric interpretation.  Consider, for simplicity, univariate random variables 
$Z^f$ and $Z^a$, then  
\begin{align*}
\mathbb{E}_{Z^{f:a}}[|z^f-z^a |^2] &= \mathbb{E}_{Z^f}[|z^f|^2] + 
\mathbb{E}_{Z^a}[|z^a |^2] - 2 \mathbb{E}_{Z^{f:a}}[z^f z^a] \\
&= \mathbb{E}_{Z^f}[|z^f|^2] + 
\mathbb{E}_{Z^a}[|z^a|^2] - 2 \mathbb{E}_{Z^{f:a}}[(z^f-\bar z^f)
(z^a-\bar z^a)] -2 \bar z^f \bar z^a \\
&= \mathbb{E}_{Z^f}[|z^f|^2] + 
\mathbb{E}_{Z^a}[|z^a|^2] -2\bar z^f \bar z^a - 2 \mbox{cov}(Z^f,Z^a),
\end{align*}
where $\bar z^{f/a} = \mathbb{E}_{Z^{f/a}}[z^{f/a}]$.
Hence finding a joint PDF $\pi_{Z^{f:a}} \in \Pi(\pi_{Z^f},\pi_{Z^a})$ that minimizes the expectation
of $|z^f-z^a|^2$ simultaneously maximizes the covariance between
$Z^f$ and $Z^a$. This geometric interpretation leads to the celebrated
Monge-Kantorovitch problem.


\begin{definition}
Let $\Pi(\pi_{Z^f},\pi_{Z^a})$ denote the set of all possible
couplings between $\pi_{Z^f}$ and $\pi_{Z^a}$.
A transference plan
  $\pi_{Z^{f:a}}^\ast \in \Pi(\pi_{Z^f},\pi_{Z^a})$ is called the solution to
  the \emph{Monge-Kantorovitch problem} with cost function $c(z^f,z^a)
  = \|z^f-z^a\|^2$ if
\begin{equation} \label{sec4:eq:MK}
\pi_{Z^{f:a}}^\ast = \arg \inf_{\pi_{Z^{f:a}}\in \Pi(\pi_{Z^f},\pi_{Z^a})} \mathbb{E}_{Z^{f:a}} [ \|z^f-z^a\|^2].
\end{equation}

The associated functional $W(\pi_{Z^f},\pi_{Z^a})$, defined by
\begin{equation} \label{sec4:WD}
W(\pi_{Z^f},\pi_{Z^a})^2 = \mathbb{E}_{Z^{f:a}} [ \|z^f-z^a\|^2]
\end{equation}
is called the \emph{$L^2$-Wasserstein distance} between $\pi_{Z^f}$ and $\pi_{Z^a}$. 
\end{definition}

\begin{example} Let us consider the discrete set
\begin{equation} \label{sec4:td}
\mathbb{Z} = \{z_1,z_2,\ldots, z_M\}, \qquad z_i \in \mathbb{R},
\end{equation}
and two probability vectors $\mathbb{P}(z_i) = 1/M$ and
$\mathbb{P}(z_i) = w_i$, respectively, on $\mathbb{Z}$ with $w_i\ge 0$, $i=1,\ldots,M$, and 
$\sum_i w_i =1$. Any coupling between these two probability vectors is characterized
by a matrix $T \in \mathbb{R}^{M\times M}$ such that its entries $t_{ij} =
(T)_{ij}$ satisfy $t_{ij} \ge 0$ and
\begin{equation} \label{sec4:margLP}
\sum_{i=1}^M t_{ij} = 1/M, \qquad \sum_{j=1}^M t_{ij} = w_i .
\end{equation}
These matrices characterize the set of all couplings $\Pi$ in the
definition of the Monge-Kantorovitch problem. Given a coupling $T$ and the mean values 
\[
\bar z^f = \frac{1}{M}\sum_{i=1}^M z_i, \qquad \bar z^a = \sum_{i=1}^M w_i z_i,
\]
the covariance between the associated discrete random variables
${\rm Z}^f:\Omega \to \mathbb{Z}$ and ${\rm Z}^a:\Omega \to \mathbb{Z}$ is defined by
\begin{equation} \label{sec4:maxLP}
\mbox{cov}({\rm Z}^f,{\rm Z}^a) = \sum_{i,j=1}^M (z_i-\bar z^a) t_{ij} (z_j-\bar z^f).
\end{equation}
The particular coupling defined by $t_{ij} = w_i/M$ leads to zero
correlation between ${\rm Z}^f$ and ${\rm Z}^a$. On the other hand, maximizing the correlation leads
to a \emph{linear transport problem} in the $M^2$ unknowns
$\{t_{ij}\}$. More precisely, the unknowns $t_{ij}$ have to satisfy
the inequality constraints $t_{ij}\ge 0$, the equality constraints
(\ref{sec4:margLP}), 
and should minimize
\[
J(T) = \sum_{i,j=1}^M t_{ij} |z_i - z_j |^2,
\]
which is equivalent to maximizing (\ref{sec4:maxLP}). See
\cite{sr:strang} and \cite{sr:wright99} for an introduction to linear
transport problems and, more generally, to linear programming. 
\end{example}


We now return to continuous random variables and the desired coupling
between forecast and analysis PDFs. The following theorem is an adaptation of
a more general result on optimal couplings from \cite{sr:Villani}.

\begin{theorem} If the forecast PDF $\pi_{Z^f}$ has bounded
  second-order moments, then the optimal
  transference plan that solves the Monge-Kantorovitch problem
  gives rise to a deterministic coupling with transport map
  \[
  Z^a = \nabla_z\phi(Z^f),
  \]
  where $\phi:\mathbb{R}^{N_z} \to
  \mathbb{R}$ is a convex potential.
\end{theorem} 

Below we sketch the basic line of arguments that lead to this
theorem. We first introduce
the \emph{support} of a coupling $\pi_{Z^{f:a}}$ on $\mathbb{R}^{N_z}
\times \mathbb{R}^{N_z}$ as the smallest closed set on which
$\pi_{Z^{f:a}}$ is concentrated,
\emph{i.e.},
\[
\mbox{supp} \,(\pi_{Z^{f:a}}) := \bigcap \{ S \subset
\mathbb{R}^{N_z}\times
\mathbb{R}^{N_z}: \,S\,
\,\mbox{closed and}\,\, \mu_{Z^{f:a}}(\mathbb{R}^{N_z} \times
\mathbb{R}^{N_z}
\setminus S) = 0 \}
\]
with the measure of $\mathbb{R}^{N_z}\times \mathbb{R}^{N_z} \setminus S$ defined by
\[
\mu_{Z^{f:a}}(\mathbb{R}^{N_z} \times
\mathbb{R}^{N_z}
\setminus S) = \int_{\mathbb{R}^{N_z} \times
\mathbb{R}^{N_z}
\setminus S}  \pi_{Z^{f:a}}(z^f,z^a) \,{\rm d}z^f {\rm d}z^a .
\]
The support of $\pi_{Z^{f:a}}$ is called \emph{cyclically monotone} if for 
every set of points $(z^f_i,z^a_i) \in \mbox{supp}\,(\pi_{Z^{f:a}})
\subset \mathbb{R}^{N_z} \times \mathbb{R}^{N_z}$,
$i=1,\ldots,I$, and any permutation $\sigma$ of $\{1,\ldots,I\}$ one
has
\begin{equation} \label{sec4:cyclical}
\sum_{i=1}^I \|z_i^f - z_i^a\|^2 \le \sum_{i=1}^I \|z_i^f -
z^a_{\sigma (i)}\|^2 .
\end{equation}
Note that (\ref{sec4:cyclical}) is equivalent to
\[
\sum_{i=1}^I (z_i^f)^T (z_{\sigma (i)}^a - z_i^a) \le 0.
\] 
It can be shown that any coupling whose support is not cyclically
monotone can be modified into another coupling with lower 
transport cost. Hence it follows that a solution $\pi^\ast_{Z^{f:a}}$
of the Monge-Kantorovitch problem (\ref{sec4:eq:MK}) has cyclically
monotone support. 

A fundamental theorem (\emph{Rockafellar's theorem}
\citep{sr:Villani}) of convex analysis now states
that cyclically monotone sets $S \subset \mathbb{R}^{N_z}\times \mathbb{R}^{N_z}$ 
are contained in the \emph{subdifferential} of a convex
function $\phi:\mathbb{R}^{N_z} \to \mathbb{R}$. Here the subdifferential
$\partial \phi$ of a convex function $\phi$ at a point $z\in \mathbb{R}^{N_z}$ is defined
as the compact, non-empty and convex set of all $m \in \mathbb{R}^{N_z}$ such that
\[
\phi(z') \ge \phi(z) + m (z'-z)
\]
for all $z' \in \mathbb{R}^{N_z}$. We write $m \in \partial
\phi(z)$. 
An optimal transport map is obtained whenever the convex
potential $\phi$ for a given optimal coupling $\pi^\ast_{Z^{f:a}}$ is sufficiently regular in
which case the subdifferential $\partial \phi (z)$ reduces to the classic
gradient $\nabla_z \phi$ and $z^a = \nabla_z \phi(z^f)$. This
regularity is ensured by the assumptions of the above theorem.
See \cite{sr:mccann95} and \cite{sr:Villani} for more details. 

We summarize the \emph{McKean optimal transportation approach} in the
following definition.

\begin{definition}
Given a dynamic iteration (\ref{sec2:DS}) with PDF $\pi_{Z^0}$ for the initial conditions and 
observations $y_{\rm obs}^n$, $n=1,\ldots,N$, the forecast $Z^{n,f}$ at
iteration index $n>0$ is defined by (\ref{sec4:McKean1}) and 
the analysis $Z^{n,a}$ by (\ref{sec4:McKean2b}). The convex potential $\phi^n$ is 
the solution to the Monge-Kantorovitch optimal transportation problem for coupling
$\pi_{Z^{f,n}}$ and $\pi_{Z^{a,n}}$. The iteration is started at $n=0$
with $Z^{0,a} = Z^0$.
\end{definition}

The application of optimal transportation to Bayesian inference and data
assimilation has first been discussed by \cite{sr:reich10},
\cite{sr:marzouk11}, and \cite{sr:cotterreich}.

In the following section we discuss data assimilation
algorithms from a McKean optimal transportation perspective. 

%
%
%

\section{Linear ensemble transform methods}
\label{sec:5}

In this section, we discuss SMCMs, EnKFs, and the recently proposed
\citep{sr:reich13} ETPF from a coupling perspective. All three data assimilation methods have in common that 
they are based on an ensemble $z_i^n$, $i=1,\ldots,$M, of model states. In the absence of
observations the $M$ ensemble members propagate independently under the model
dynamics (\ref{sec2:DS}), \emph{i.e.}, an analysis ensemble at time-level $n-1$ gives rise to a 
forecast ensemble at time-level $n$ via
\[
z_i^{n,f} = \Psi(z_i^{n-1,a}), \qquad i=1,\ldots,M.
\] 
The three methods differ in the way the forecast ensemble $\{z_i^{n,f}\}_{i=1}^M$ 
is transformed into an analysis ensemble $\{z_i^{n,a}\}_{i=1}^M$ 
under an observation $y_{\rm obs}^n$. However, all three methods share a common
mathematical structure which we outline next. We drop the iteration
index in order to simplify the notation. 

\begin{definition}
The class of \emph{linear ensemble transform filters} (LETFs) is defined by
\begin{equation} \label{sec5:transform}
z^{a}_j = \sum_{i=1}^M z_i^{f} s_{ij},
\end{equation}
where the coefficients $s_{ij}$ are the $M^2$ entries of a matrix 
$S \in \mathbb{R}^{M\times M}$. 
\end{definition}

The concept of LETFs is well established for EnKF formulations \citep{sr:tippett03}. 
It will be shown below that SMCMs and the ETPF also belong to the class of LETFs.  
In other words, these three methods differ only in the definition of the corresponding transform matrix $S$. 


\subsection{Sequential Monte Carlo methods (SMCMs)} \label{sec51}

A central building block of a SMCM is the
\emph{proposal density} $\pi_{\rm prop}(z|z')$, which produces a \emph{proposal
ensemble} $\{z^p_i\}_{i=1}^M$ from the last analysis ensemble. Here
we assume, for simplicity, that the proposal density is given by the model dynamics
itself, \emph{i.e.},
\[
\pi_{\rm prop}(z|z') = \delta (z-\Psi(z')),
\]
and then
\[
z_i^p = z_i^f ,\qquad i=1,\ldots,M.
\]
One associates with the proposal/forecast ensemble two
discrete measures on 
\begin{equation} \label{sec5:SMCM1a}
\mathbb{Z} =
\{z_1^{f},z_2^{f},\ldots,z_M^{f} \}, 
\end{equation}
namely the uniform measure $\mathbb{P}(z_i^{f}) = 1/M$ and the non-uniform measure
\[
\mathbb{P}(z_i^{f}) = w_i, 
\]
defined by the importance weights
\begin{equation} \label{sec5:SMCM1b}
w_i = \frac{\pi_Y(y_{\rm
    obs}|z_i^{f})} {\sum_{j=1}^M \pi_Y(y_{\rm obs}| z_j^{f})}.
\end{equation}
The \emph{sequential importance resampling} (SIR) filter \citep{sr:gordon93}
resamples from the weighted forecast ensemble in order to produce a new equally weighted 
analysis ensemble $\{z_i^a\}$. Here we only consider SIR filter implementations
with resampling performed after each data assimilation cycle.

An in-depth discussion of the SIR filter and more general SMCMs 
can be found, for example, in \cite{sr:Doucet,sr:doucet11}. Here we focus
on the coupling of discrete measures perspective of a resampling step.
We first note that any resampling strategy effectively leads to a coupling of the 
uniform and the non-uniform measure on (\ref{sec5:SMCM1a}). 
As previously discussed, a coupling is defined by a matrix $T \in
\mathbb{R}^{M\times M}$ such that $t_{ij} \ge 0$, and
\begin{equation} \label{sec5:SMCM1c}
\sum_{i=1}^M t_{ij} = 1/M, \qquad \sum_{j=1}^M t_{ij} = w_i.
\end{equation}
The resampling strategy (\ref{sec4:DelMoral}) leads to
\[
t_{ij} = \frac{1}{M}\left(\epsilon w_j \delta_{ij} + (1- \epsilon
  w_j)w_ i \right)
\]
with $\epsilon \ge 0$ chosen such that $\epsilon w_j \le 1$ for all $j=1,\ldots,M$.
Monomial resampling corresponds to the special case $\epsilon = 0$, \emph{i.e.}~$t_{ij} = w_i/M$.
The associated transformation matrix $S$ in (\ref{sec5:transform}) is the realization of a random matrix
with entries $s_{ij} \in \{0,1\}$ such that each column of $S$
contains exactly one entry $s_{ij} = 1$. Given a coupling $T$, the
probability for the $s_{ij}$'s entry to be the one selected in the
$j$th column is
\[
\mathbb{P}(s_{ij} = 1) = M t_{ij} 
\]
and $Mt_{ij} = w_i$ in case of monomial resampling. Any such resampling procedure 
based on a coupling matrix $T$ leads to a consistent coupling for the underlying
forecast and analysis probability measures as $M\to \infty$, which, however, is non-optimal
in the sense of the Monge-Kantorovitch problem (\ref{sec4:eq:MK}). 
We refer to \cite{sr:crisan} for resampling strategies which satisfy alternative optimality conditions.
  
We emphasize that the transformation matrix $S$ of a SIR particle filter analysis step 
satisfies 
\begin{equation} \label{sec5:LETF1a}
\sum_{i=1}^M s_{ij} = 1
\end{equation}
and
\begin{equation} \label{sec5:LETF1b}
s_{ij}\in [0,1].
\end{equation}
In other words, each realization of the resampling step yields a Markov
chain $S$. Furthermore, the weights $\hat w_i = M^{-1} \sum_{j=1}^M
s_{ij}$ satisfy $\mathbb{E}[\hat w_i] = w_i$ and the analysis ensemble
defined by 
$z_j^a = z_i^f$ if $s_{ij} = 1$, $j=1,\ldots,M$, is contained in the convex hull
of the forecast ensemble (\ref{sec5:SMCM1a}).

A forecast ensemble $\{z_i^f\}_{i=1}^M$ leads to the following
estimator 
\[
\bar z^f = \frac{1}{M} \sum_{i=1}^M z_i^f
\]
for the mean and
\[
P_{zz}^f = \frac{1}{M-1} \sum_{i=1}^M (z_i^f-\bar z^f)(z_i^f -\bar
z^f)^T
\]
for the covariance matrix. In order to increase the \emph{robustness}
of a SIR particle filter one often augments the resampling
step by the \emph{particle rejuvenation step} \citep{sr:pham01}
\begin{equation} \label{sec5:rejuvenation1}
z_j^a = z_i^f + \xi_j, 
\end{equation}
where the $\xi_j$'s are realizations of $M$ independent and
identically distributed Gaussian random variables ${\rm
  N}(0,h^2P_{zz}^f)$ and $s_{ij} = 1$. Here $h>0$ is the \emph{rejuvenation parameter} which
determines the magnitude of the stochastic perturbations. Rejuvenation
helps to avoid the creation of identical analysis ensemble members
which would remain identical under the deterministic model dynamics
(\ref{sec2:DS}) for all times. Furthermore, rejuvenation can be used
as a heuristic tool in order to compensate for model errors as encoded, for example,
in the difference between (\ref{sec2:DS}) and (\ref{sec2:DST}).

In this paper we only discuss SMCMs which are based on the proposal
step (\ref{sec1:forecast}). Alternative proposal steps are possible
and recent work on alternative implementations of SMCMs include
\cite{sr:leeuwen10}, \cite{sr:chorin10}, \cite{sr:chorin12},
\cite{sr:chorin12b}, \cite{sr:leeuwen12}, \cite{sr:reich13b}.


\subsection{Ensemble Kalman filter (EnKF)}

The historically first version of the EnKF uses perturbed observations
in order to transform a forecast ensemble into an analysis ensemble. The key requirement of any
EnKF is that the transformation step is consistent with the classic Kalman update step in
case the forecast and analysis PDFs are Gaussian. The, so called, \emph{EnKF with perturbed
observations} is explicitly given by the simply formula
\[
z_j^a = z_j^f - K (y_j^f + \xi_j - y_{\rm obs}), \qquad j=1,\ldots,M,
\]
where $y_j^f = h(z_j^f)$, the $\xi_j$'s are realizations of independent and identically distributed
Gaussian random variables with mean zero and covariance matrix $R$, and $K$ denotes 
the Kalman gain matrix, which in case of the EnKF is determined by the forecast ensemble as follows:
\[
K = P_{zy}^f (P_{yy}^f + R)^{-1}
\]
with empirical covariance matrices
\[
P^f_{zy} = \frac{1}{M-1} \sum_{i=1}^M (z_i^f - \bar z^f)(y_i^f - \bar y^f)^T
\]
and
\[
P^f_{yy} = \frac{1}{M-1} \sum_{i=1}^M (y_i^f - \bar y^f)(y_i^f - \bar y^f)^T,
\]
respectively.  Here the ensemble mean in observation space is defined
by
\[
\bar y^f = \frac{1}{M} \sum_{i=1}^M y_i^f.
\]

In order to shorten subsequent notations, we introduce the 
$N_y \times M$ matrix of ensemble deviations 
\[
A_y^f = [ y_1^f-\bar y^f, y_2^f -\bar y^f, \dots, y_M^f-\bar y^f] 
\]
in observation space and the $N_z \times M$ matrix of ensemble deviations
\[
A_z^f = [ z_1^f-\bar z^f, z_2^f -\bar z^f, \dots, z_M^f-\bar z^f] 
\]
in state space, respectively. 
In terms of these ensemble deviation matrices, it holds
that 
\[
P_{zz}^f = \frac{1}{M-1} A_z^f (A_z^f)^T \quad \mbox{and} \quad
P_{zy}^f = \frac{1}{M-1} A_z^f (A_y^f)^T,
\]
respectively.

It can be verified by explicit calculations that the EnKF 
with perturbed observations fits into the class of LETFs with
\[
z_j^a = \sum_{i=1}^M z_i^f \left( \delta_{ij}  - \frac{1}{M-1} (y_i^f -\bar y_i^f)^T
(P_{yy}^f + R)^{-1} ( y_j^f + \xi_j - y_{\rm obs}) \right)
\]
and, therefore,
\[
s_{ij} = \delta_{ij} - \frac{1}{M-1}(y_i^f - \bar y^f)^T (P_{yy}^f + R)^{-1} 
(y_j^f + \xi_j - y_{\rm obs}).
\]
Here we have used that
\[
\frac{1}{M-1} \sum_{i=1}^M (z_i^f - \bar z^f)(y_i^f - \bar y^f)^T = 
\frac{1}{M-1} \sum_{i=1}^M z_i^f (y_i^f - \bar y^f)^T.
\]
We note that the transform matrix $S$ is the realization of a random matrix. The class of 
\emph{ensemble square root filters} (ESRF) leads instead to
deterministic transformation matrices $S$.
More precisely, an ESRF uses separate transformation steps for the ensemble mean $\bar z^f$ and
the ensemble deviations $z_i^f - \bar z^f$. The mean is simply updated according to
the classic Kalman formula, \emph{i.e.}
\begin{equation} \label{sec5:Kalmanmean}
\bar z^a = \bar z^f - K(\bar y^f - y_{\rm obs})
\end{equation}
with the Kalman gain matrix defined as before. 

Upon introducing the analysis matrix of ensemble deviations $A_z^a \in
\mathbb{R}^{N_z\times M}$, one obtains
\begin{align} \nonumber 
P_{zz}^a &= \frac{1}{M-1} A_z^a (A_z^a)^T \\
&= P_{zz}^f - K (P_{zy}^f)^T = \frac{1}{M-1} A_z^f Q (A_z^f)^T \label{sec5:Kalmancovariance}
\end{align}
with the $M\times M$ matrix $Q$ defined by
\[
Q = I - \frac{1}{M-1} (A^f_y)^T (P_{yy}^f + R)^{-1} A_y^f .
\]
Let us denote the matrix square root\footnote{The matrix square root
  of a symmetric positive semi-definite matrix $Q$ is the unique
  symmetric matrix $D$ which satisfies $D D = Q$.} of $Q$ by $D$ and its entries by
$d_{ij}$. 

We note that $\sum_{i=1}^M d_{ij} = 1$ and it follows that
\begin{align} \nonumber
z_j^a &= \bar z^f - K(\bar y^f - y_{\rm obs}) + \sum_{i=1}^M (z_i^f
-\bar z^f) d_{ij} \\ \nonumber
&= \sum_{i=1}^M z_i^f  \left(
\frac{1}{M-1}(y_i^f -\bar y^f)^T(P_{yy}^f + R)^{-1} (y_{\rm obs} - \bar y^f)  + d_{ij} \right)\\
&= \sum_{i=1}^M z_i^f \left( \frac{1}{M-1}(y_i^f -\bar y^f)^T(P_{yy}^f + R)^{-1} 
(y_{\rm obs} - \bar y^f ) + d_{ij}  \right) . \label{sec5:ESRF}
\end{align}
The appropriate entries for the transformation matrix $S$ of an ESRF
can now be read off of (\ref{sec5:ESRF}). See 
\cite{sr:tippett03,sr:wang04,sr:nichols08,sr:ott04,sr:nerger12} and
\cite{sr:evensen} for further details on other ESRF implementations such as the ensemble
adjustment Kalman filter. We mention in particular that an application of the 
Sherman-Morrison-Woodbury formula
\citep{sr:golub} leads to the equivalent square root formula
\begin{align} 
D &= \left\{ I +   \frac{1}{M-1} (A_y^f)^T R^{-1}  A_y^f  \right\}^{-1/2}, \label{sec5:eq:Stransform}
\end{align}
which avoids the need for inverting the $N_y \times N_y$ matrix
$P_{yy}^f + R$, which is desirable whenever $N_y \gg M$. 
Furthermore, using the equivalent Kalman gain matrix representation
\[
K = P^a_{zy}R^{-1},
\]
the Kalman update formula
(\ref{sec5:Kalmanmean}) for the mean becomes
\begin{align*}
\bar z^a &= \bar z^f - P^a_{zy} R^{-1} (\bar y^f - y_{\rm obs})\\
&= \bar z^f + \frac{1}{M-1} A_z^f Q (A^f_y)^T R^{-1} (y_{\rm obs} - \bar y^f).
\end{align*}
This reformulation gives rise to 
\begin{equation} \label{sec5:letkf}
s_{ij} = \frac{1}{M-1} q_{ij} (y^f_j - \bar y^f) R^{-1}
(y_{\rm obs} - \bar y^f) + d_{ij} ,
\end{equation}
which forms the basis of the \emph{local ensemble transform Kalman filter}
(LETKF) \citep{sr:ott04,sr:hunt07} to be discussed in more detail in
Section \ref{sec54}.

We mention that the EnKF with perturbed
observations or an ESRF implementation leads to transformation
matrices $S$ which satisfy (\ref{sec5:LETF1a}) but the entries
$s_{ij}$ can take positive as well as negative values. This can be
problematic in case the state variable $z$ should be
non-negative. Then it is possible that a forecast ensemble $z_i^f \ge
0$, $i=1,\ldots,M$, is transformed into an analysis $z_i^a$, which
contains negative entries. See \cite{sr:janjic13} for modifications to 
EnKF type algorithms in order to preserve positivity. 

One can discuss the various EnKF formulations from an optimal
transportation perspective. Here the coupling is between two Gaussian
distributions; the forecast PDF ${\rm N}(\bar z^f,P_{zz}^f)$
and analysis PDF ${\rm N}(\bar z^a,P_{zz}^a)$, respectively with the
analysis mean given by (\ref{sec5:Kalmanmean}) and the analysis
covariance matrix by (\ref{sec5:Kalmancovariance}).
We know that the optimal coupling must be of the form
\[
z^a = \nabla_z \phi (z^f)
\]
and, in case of Gaussian PDFs, the convex potential
$\phi:\mathbb{R}^{N_z} \to \mathbb{R}$ is
furthermore bilinear, \emph{i.e.},
\[
\phi(z) = b^T z + \frac{1}{2} z^T A z
\]
with the vector $b$ and the matrix $A$ appropriately defined. The choice
\[
z^a = b + Az^f = \bar z^a + A(z^f-\bar z^f)
\]
leads to 
\[
b = \bar z^a - A\bar z^f
\]
for the vector $b \in \mathbb{R}^{N_z}$. The matrix $A \in
\mathbb{R}^{N_z \times N_z}$ then needs to satisfy
\[
P_{zz}^a = AP_{zz}^f A^T.
\]
The optimal, in the sense of Monge-Kantorovitch with cost function
$c(z^f,z^a) = \|z^f-z^a\|^2$,  matrix $A$ is given by
\[
A = (P_{zz}^a)^{1/2} \left[ (P_{zz}^a)^{1/2} P_{zz}^f (P_{zz}^a)^{1/2}
\right]^{-1/2} (P_{zz}^a)^{1/2} .
\]
See \cite{sr:olkin82}. An efficient implementation of this optimal
coupling in the context of ESRFs has been discussed in
\cite{sr:cotterreich}. The essential idea is to replace the matrix
square root of $P_{zz}^a$ by the analysis matrix of
ensemble deviations  $A_z^a = A_z^f D$ scaled by $1/\sqrt{M-1}$.

Note that different cost functions $c(z^f,z^a)$ lead to different
solutions to the associated Monge-Kantorovitch problem
(\ref{sec4:eq:MK}). Of particular interest is the weighted inner
product
\[
c(z^f,z^a) = \left( (z^f-z^a)^T B^{-1} (z^f - z^a)\right)^2
\]
for an appropriate positive definite matrix $B \in \mathbb{R}^{N_z
  \times N_z}$ \citep{sr:cotterreich}. 

As for SMCMs particle rejuvenation can be
applied to the analysis from an EnKF or ESRF. However, the more
popular method for increasing the robustness of EnKFs is to apply 
\emph{multiplicative ensemble inflation} 
\begin{equation} \label{sec5:inflation}
z_i^f \to \bar z^f + \alpha (z_i^f - \bar z^f), \qquad \alpha\ge 1,
\end{equation}
to the forecast ensemble prior to the application of an EnKF or ESRF.
The parameter $\alpha$ is called the inflation factor. An adaptive
strategy for determining the factor $\alpha$ has, for example, been
proposed by \cite{sr:anderson07,sr:miyoshi11}. The inflation factor
$\alpha$ can formally be related to the rejuvenation parameter $h$ 
in (\ref{sec5:rejuvenation1})
through
\[
\alpha = \sqrt{1+h^2}.
\]
This relation becomes exact as $M\to \infty$. 

We mention that the \emph{rank histogram filter} of
\cite{sr:anderson10}, which uses a nonlinear filter in observation space and linear regression
from observation onto state space, also fits into the framework of the LETFs. See
\cite{sr:reichcotter14} for more details. The \emph{nonlinear ensemble adjustment filter} 
of \cite{sr:lei11}, on the other hand, falls outside the class of LETFs.


\subsection{Ensemble transform particle filter (ETPF)} \label{secETPF}

We now return to the SIR filter described in
Section \ref{sec51}. Recall that a SIR filter relies on
importance resampling which we have interpreted as a coupling
between the uniform measure on (\ref{sec5:SMCM1a}) and the measure
defined by (\ref{sec5:SMCM1b}). Any coupling is characterized by a
matrix $T$ such that its entries are non-negative and
(\ref{sec5:SMCM1c}) hold. 

\begin{definition} \label{sec5:ETPF}
The ETPF is based on choosing the $T$
which minimizes
\begin{equation} \label{sec5:functional1}
J(T) = \sum_{i,j=1}^M t_{ij} \|z_i^f  - z_j^f\|^2
\end{equation}
subject to (\ref{sec5:SMCM1c}) and $t_{ij}\ge 0$.
Let us denote the minimizer by $T^\ast$. Then the transform matrix $S$
of an ETPF is defined by
\[
S = MT^\ast,
\]
which satisfies (\ref{sec5:LETF1a}) and (\ref{sec5:LETF1b}).
\end{definition}

Let us give a geometric interpretation of the ETPF transformation step. 
Since $T^\ast$ from Definition \ref{sec5:ETPF} provides an optimal coupling, 
Rockafellar's theorem implies the existence of a convex potential
$\phi_M :\mathbb{R}^{N_z} \to
\mathbb{R}$ such that 
\[
z_i^f \in \partial \phi_M (z_j^f) \quad \mbox{for all}\quad i \in
{\cal I}_j :=  \{i' \in \{1,\ldots,M\}: t^\ast_{i'j} > 0\},
\]
$j=1,\ldots,M$. In fact, $\phi_M$ can be chosen to be piecewise affine and a constructive 
formula can be found in \cite{sr:Villani}. The ETPF transformation step
\begin{equation} \label{sec5:etpftransform}
z_j^a = M \sum_{i=1}^M z_i^f t^\ast_{ij} = \sum_{i=1}^M z_i^f s_{ij}
\end{equation}
corresponds to a particular selection from the linear space 
$\partial \phi_M (z_j^f)$, $j=1,\ldots,M$; namely the 
expectation value of the discrete random variable 
\[
{\rm Z}_j^a: \Omega \to \{z_1^f,z_2^f,\ldots,z_M^f\}
\]
with probabilities $\mathbb{P}(z_i^f) = s_{ij}$, $i=1,\ldots,M$. 
Hence it holds that
\[
\bar z^a := \frac{1}{M}\sum_{j=1}^M z_j^a = \sum_{i=1}^M w_i z_i^f .
\]
See \cite{sr:reich13} for more details. where 
it has also been shown that the potentials $\phi_M$ converge to the solution of the
underlying continuous Monge-Kantorovitch problem as the ensemble size
$M$ approaches infinity. 


It should be noted that standard algorithms for finding the minimizer of (\ref{sec5:functional1}) 
suffer from a ${\cal O}(M^3 \log M)$ computational complexity. This complexity has been
reduced to ${\cal O}(M^2 \log M)$ by \cite{sr:Pele-iccv2009}. There are also
fast iterative methods for finding approximate minimizers of (\ref{sec5:functional1}) using
the \emph{Sinkhorn distance} \citep{sr:cuturi13a}.

The particle rejuvenation step (\ref{sec5:rejuvenation1}) for SMCMs can be
extended to the ETPF as follows:
\begin{equation} \label{sec5:rejuvenation2}
z_j^a = \sum_{i=1}^M z_i^f s_{ij} + \xi_j, \qquad j=1,\ldots,M.
\end{equation}
As before the $\xi_j$'s are realizations of $M$ independent Gaussian
random variables with mean zero and appropriate covariance matrices $P_j^a$. We use 
$P_j^a = h^2 P_{zz}^f$ with rejuvenation parameter $h>0$ for the 
numerical experiments conducted in this paper. Another possibility
would be to locally estimate $P_j^a$ from the coupling matrix
$T^\ast$, \emph{i.e.},
\[
P_j^a = \sum_{i=1}^M s_{ij} (z_i^f - \bar z_j^a)(z_i^f - \bar z_j^a)^T
\]
with mean $\bar z_j^a = \sum_{i=1}^M s_{ij} z_i^f$.


\subsection{Quasi-Monte Carlo (QMC) convergence} \label{sec55}
 
The expected rate of convergence for standard Monte Carlo methods
is $M^{-1/2}$ where $M$ denotes the ensemble size. 
QMC methods have an upper bound of $\log(M)^dM^{-1}$ where $d$
stands for the dimension \citep{sr:caflisch88}. For the purpose of this paper, $d=N_z$. 
Unlike Monte Carlo methods QMC methods also depend on the dimension of the space which 
implies a better performance for small $N_z$ or/and large $M$. However, in practice QMC
methods perform considerably better than the theoretical bound for the
convergence rate and outperform Monte Carlo methods even for small ensemble 
sizes and in very high dimensional models. The latter may be explained by the concept of
\textit{effective dimension} introduced by \cite{sr:caflisch97}.

The following simulation investigates the convergence rate of the
estimators for the first and second moment of the posterior
distribution after applying a single analysis step of a SIR particle filter
and an ETPF. The prior is chosen to be a uniform
distribution on the unit square and the sum of
both components is observed with additive noise drawn from a centered
Gaussian distribution with variance equal to two.
Reference values for the posterior moments are generated using Monte Carlo
importance sampling with sample size $M=2^{26}$. QMC samples of different sizes 
are drawn from the prior distribution and a single residual resampling step is
compared to a single transformation step using an optimal coupling $T^\ast$. 
Fig.~\ref{fig_QMC} shows the root mean square errors (RMSEs) of the different posterior
estimates with respect to their reference values. We find that the transform method
preserves the optimal $M^{-1}$ convergence rate of the prior QMC samples while
resampling reduces the convergence rate to the $M^{-1/2}$.

\begin{figure}[h]
    \centering
  \subfigure{\includegraphics[clip=TRUE,scale=.3,trim =.5cm 0cm .5cm 0cm]{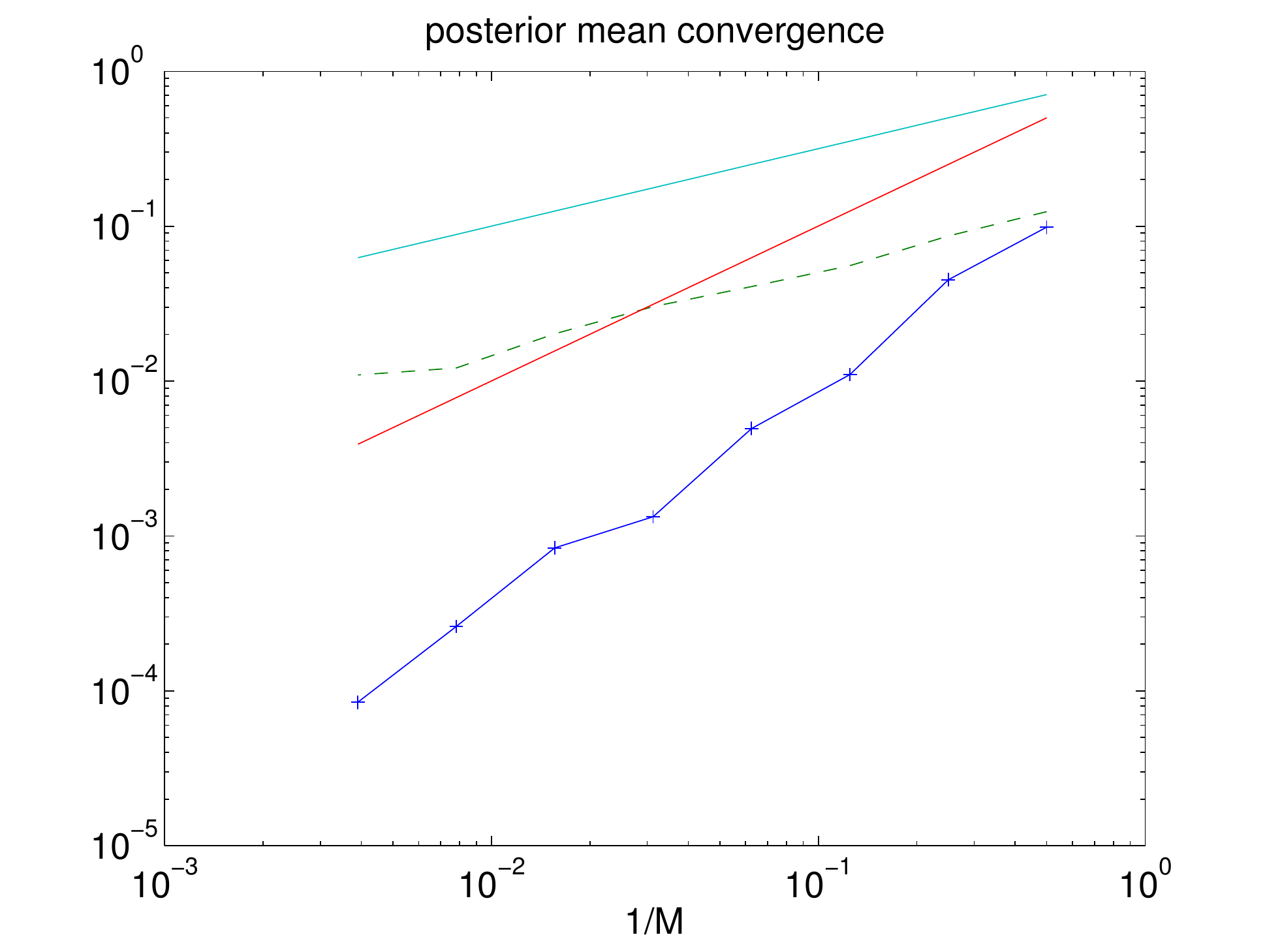}}
  \hfill
  \subfigure{\includegraphics[clip=TRUE,scale=.3,trim =.5cm 0cm .5cm 0cm]{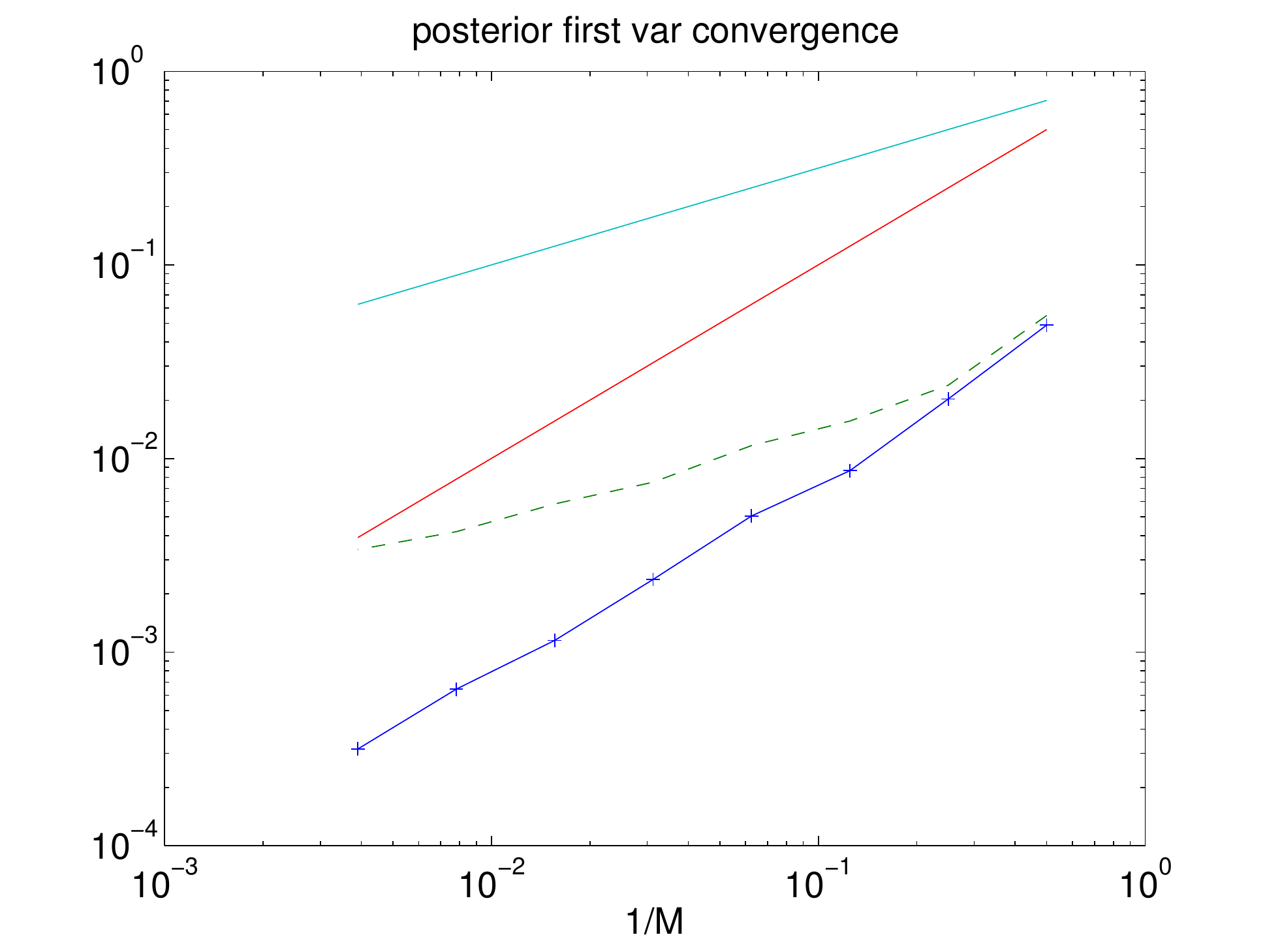}}
  \hfill
  \subfigure{\includegraphics[clip=TRUE,scale=.3,trim =0.5cm 0cm .5cm
    0cm]{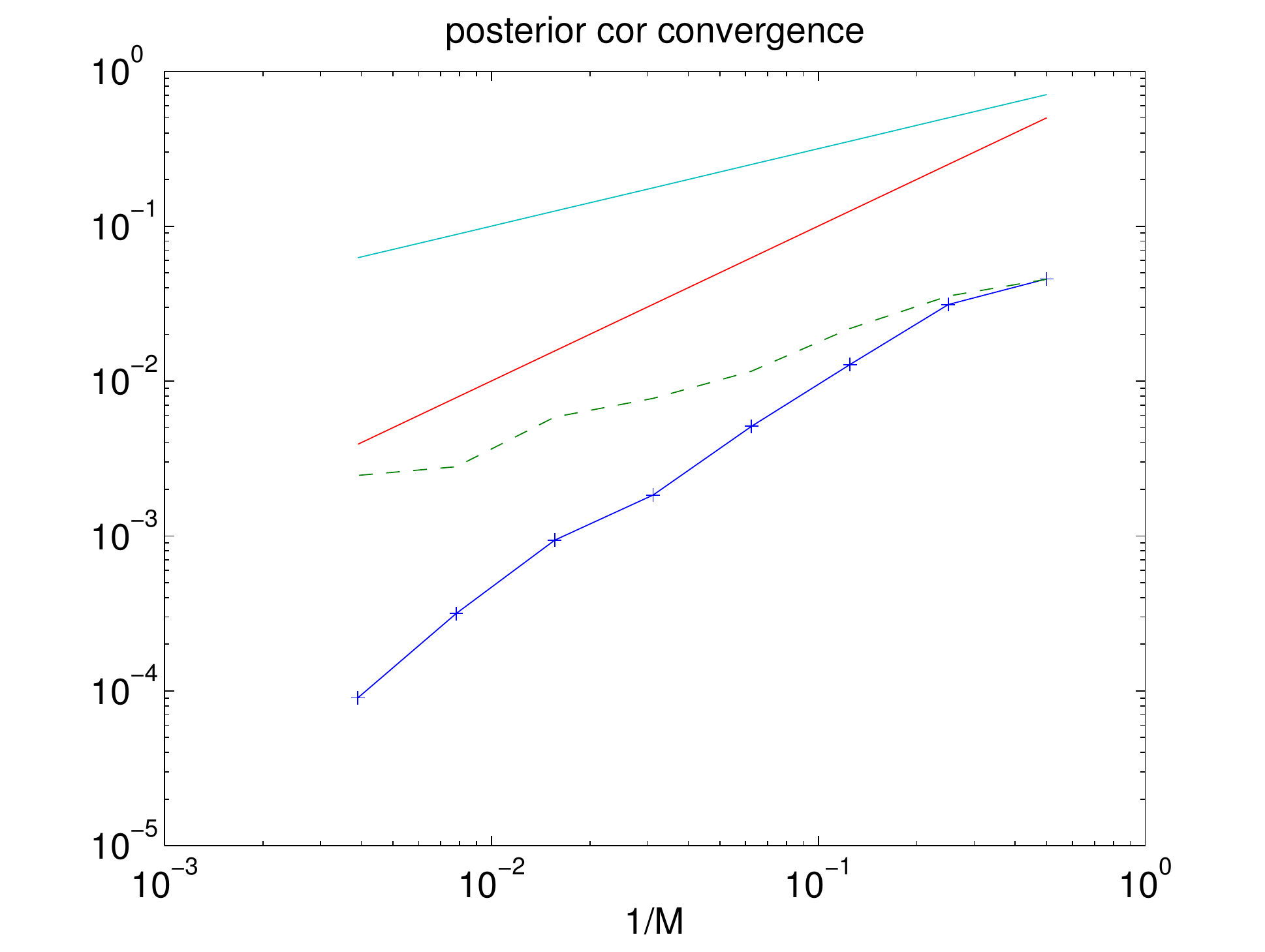}}
  \hfill
  \subfigure{\includegraphics[clip=TRUE,scale=.3,trim =0.5cm 0cm .5cm
    0cm]{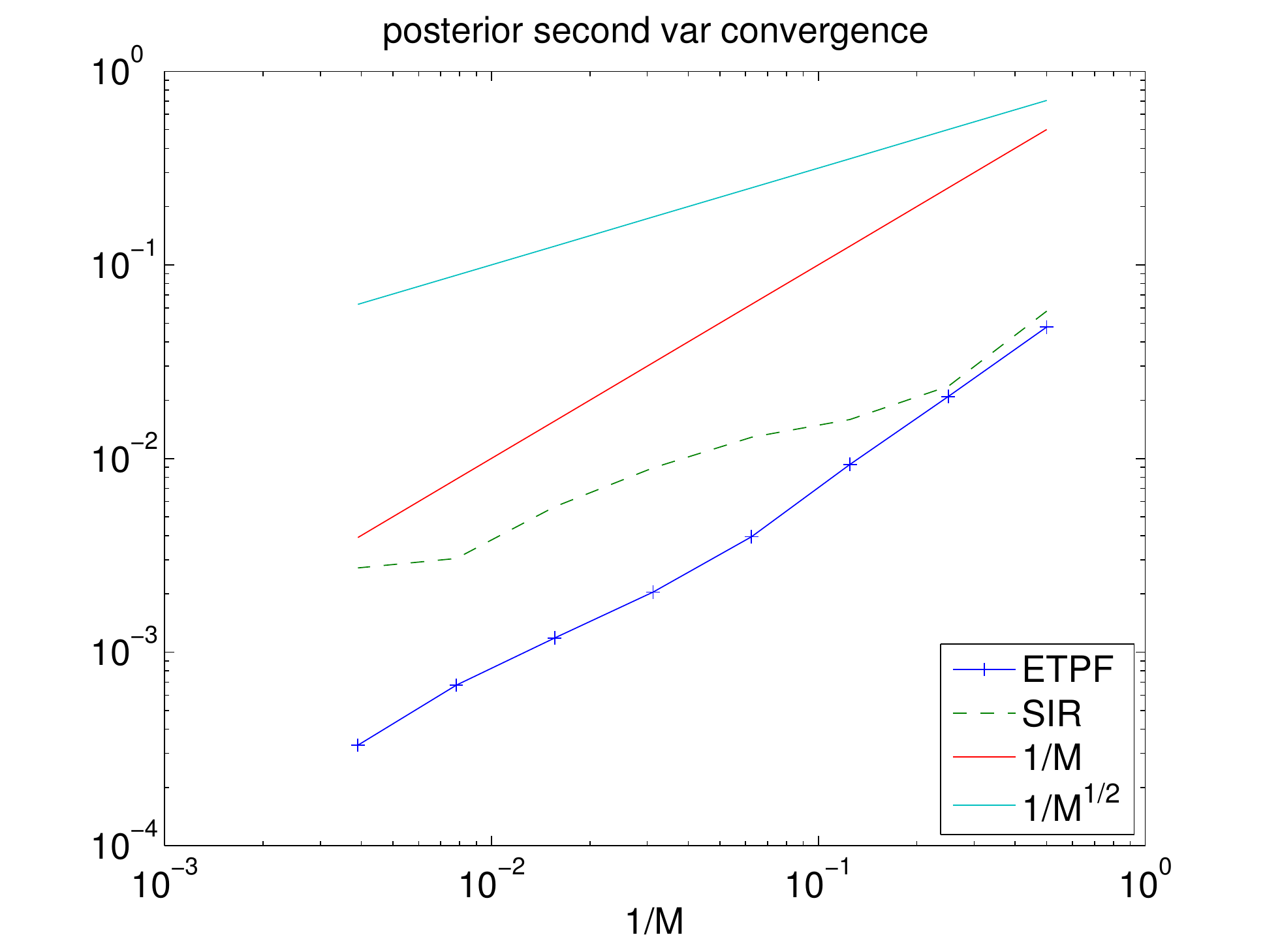}}
  \caption{RMSEs of estimates for the posterior mean, variances (var), and correlation (cor)
    using importance resampling (SIR) and optimal transformations (ETPF)
    plotted on a log-log scale as a function of ensemble sizes $M$.}
  \label{fig_QMC}
\end{figure}

We mention that replacing the deterministic transformation step in
(\ref{sec5:etpftransform}) by drawing ensemble member $j$ from the
prior ensemble according to the weights given by the $j$-th column of
$S$ leads to a stochastic version of the ETPF. This variant, despite
being stochastic like the importance resampling step, results again in a QMC
convergence rate. 

%
%
%

\section{Spatially extended dynamical systems and localization} \label{sec54}

Let us start this section with a simple thought experiment on the curse of dimensionality.
Consider a state space of dimension $N_z = 100$ and a prior Gaussian distribution
with mean zero and covariance matrix $P^f = I$. The reference solution is $z_{\rm ref} = 0$
and we observe every component of the state vector subject to independent measurement
errors with mean zero and variance $R = 0.16$. If one applies a single importance resampling
step to this problem with ensemble size $M=10$, one finds that the effective sample size
collapses to $M_{\rm eff} \approx 1$ and the resulting analysis ensemble is unable to recover 
the reference solution. However, one also quickly realizes that the stated problem 
can be decomposed into $N_z$ independent data assimilation problems in each component
of the state vector alone. If importance resampling is now performed in each component of the
state vector independently, then the effective sample size for each of the $N_z$ analysis problems
remains close to $M =10$ and the reference solution can be recovered from the
given set of observations. This is the idea of \emph{localization}. Note that localization 
has increased the total sample size to $M \times N_z = 1000$ for this problem!

We now formally extend LETFs to spatially extended systems which may
be viewed as an infinite-dimensional dynamical system
\citep{sr:robinson} and formulate an appropriate localization strategy. 
Consider the linear advection equation
\[
u_t + u_x = 0
\]
as a simple example for such a scenario. If $u_0(x)$ denotes the solution
at time $t=0$, then 
\[
u(x,t) = u_0(x+t)
\]
is the solution of the linear advection equation for all $t\ge
0$. Given a time-increment $\Delta t>0$, the associated dynamical
system maps a function $u(x)$ into $u(x+\Delta t)$. A
finite-dimensional dynamical system is obtained by discretizing in
space with mesh-size $\Delta x>0$. For example, the Box scheme
\citep{sr:Morton} leads to
\[
\frac{u_j^{k+1}+u_{j+1}^{k+1} - u_j^k-u_{j+1}^k}{2\Delta t} +
\frac{u_{j+1}^{k+1} + u_{j+1}^k - u_j^{k+1} - u_j^k}{2\Delta x} = 0
\]
and, for $J$ spatial grid points, the state vector at $t_k = k\Delta t$ becomes
\[
z^k = (u_1^k,u_2^k,\ldots,u_J^k)^T \in \mathbb{R}^J.
\]
We may take the formal limit $J\to \infty$ and $\Delta x\to 0$ in
order to return to functions $z^k(x)$. The dynamical system (\ref{sec2:DS})
is then defined as the map that propagates such functions (or their finite-difference 
approximations)  from one observation instance to the next
in accordance with the specified numerical method. Here we assume that observations
are taken in intervals of $\Delta t_{\rm obs} = N_{\rm out} \Delta t$ with 
$N_{\rm out} \ge 1$ a fixed integer. The index $n\ge 1$ in (\ref{sec2:DS}) 
is the counter for those observation instances.

In other words, forecast or analysis ensemble
members, $z^{f/a}(x)$, now become functions of $x\in \mathbb{R}$, belong to some
appropriate function space ${\cal H}$, and the dynamical system
(\ref{sec2:DS}) is formally replaced by a map or evolution equation 
on ${\cal H}$ \citep{sr:robinson}. For simplicity of exposition we
assume periodic boundary conditions, \emph{i.e.},~$z(x) = z(x+L)$ for
some appropriate $L>0$.

The curse of dimensionality \citep{sr:bengtsson08} implies that,
generally speaking, none of the LETFs discussed so far is suitable for
data assimilation of spatially extended systems. In order to overcome
this situation, we now discuss the
concept of \emph{localization} as first introduced in 
\cite{sr:houtekamer01,sr:houtekamer05} for EnKFs. While we will focus on a particular
localization, called \emph{R-localization}, suggested by \cite{sr:hunt07}, our methodology
can be extended to \emph{B-localization} as proposed by \cite{sr:hamill01}.

In the context of the LETFs R-localization amounts to modifying (\ref{sec5:transform}) to
\[
z^{a}_j(x) = \sum_{i=1}^M z_i^{f}(x) s_{ij}(x),
\]
where the associated transform matrices $S(x) \in \mathbb{R}^M \times \mathbb{R}^M$ 
depend now on the spatial location $x \in [0,L]$. It is crucial that
the transformation matrices $S(x)$ are sufficiently smooth in $x$ in
order to produce analysis ensembles with sufficient regularity for the
evolution problem under consideration and, in particular $z^a_j \in
{\cal H}$. In case of an SMCM with importance resampling, the resulting $S(x)$
would, in general, not even be continuous for almost all $x\in [0,L)$. Hence we only discuss
localization for the ESRF and the ETPF. 

Let us, for simplicity, assume that the forward
operator $h:{\cal H} \to \mathbb{R}^{N_y}$ for the observations $y_{\rm
  obs}$ is defined by
\[
h_k(z) = z({\rm x}_k), \qquad k=1,\ldots,N_y.
\]
Here the ${\rm x}_k\in [0,L)$ denote the spatial location at which the
observation is taken. The measurement errors are Gaussian with mean
zero and covariance matrix $R \in \mathbb{R}^{N_y\times N_y}$. We assume
for simplicity that $R$ is diagonal. 

In the sequel we assume that $z(x)$ has been extended to $x\in
\mathbb{R}$ by periodic extension from $x\in [0,L)$ and introduce 
time-averaged and normalized \emph{spatial correlation function}
\begin{equation} \label{sec5:spatialcorr}
C(x,s) := \frac{\sum_{n=0}^N  z^n(x+s)z^n(x) }{\sum_{n=0}^N (z^n(x))^2}
\end{equation}
for $x \in [0,L)$ and $s \in [-L/2,L/2)$. Here we have assumed that the underlying solution
process is stationary ergodic. In case of spatial homogeneity the spatial
correlation function becomes furthermore independent of $x$ for $N$ sufficiently large.

We also introduce a localization kernel ${\cal K}(x,x';r_{\rm loc})$ in order to define
$R$-localization for an ESRF and the ETPF. The localization kernel can be as simple as
\begin{equation} \label{SEC5:locrho1}
{\cal K}(x,x;r_{\rm loc})  = \left\{ \begin{array}{ll} 1-\frac{1}{2}
    s & \mbox{for} \,\, s \le 2,\\
0 & \mbox{else}, \end{array} \right.
\end{equation}
with 
\[
s := \frac{\min \{|x - x' -L|,|x-x'|,|x-x'+L|\}}{r_{\rm loc}} \ge 0,
\]
or a higher-order polynomial such as
\begin{equation} \label{sec5:locrho2}
{\cal K}(x,x';r_{\rm loc}) = \left\{ \begin{array}{ll} 1 - \frac{5}{3} s^2 + \frac{5}{8} s^3 + \frac{1}{2} s^4 -
\frac{1}{4} s^5  & \mbox{for} \,\, s \le 1,\\
-\frac{2}{3} s^{-1} + 4 - 5s + \frac{5}{3} s^2 + \frac{5}{8} s^3 - \frac{1}{2} s^4 + 
\frac{1}{12} s^5 & \mbox{for}\,\, 1 \le s \le 2,\\
0 & \mbox{else} . 
\end{array} \right. .
\end{equation}
See \cite{sr:gaspari99}.

In order to compute the transformation matrix
$S(x)$ for given $x$, we modify the $k$th diagonal entry in the
measurement error covariance matrix $R \in \mathbb{R}^{N_y\times N_y}$ and define
\begin{equation} \label{sec5:localization1}
\frac{1}{\tilde r_{kk} (x)} := \frac{{\cal K}(x,{\rm x}_k;r_{{\rm loc},R})}{r_{kk}}
\end{equation}
for $k=1,\ldots,N_y$. Given a localization radius $r_{{\rm loc},R}>0$,
this results in a matrix $\tilde R^{-1}(x)$ which replaces the $R^{-1}$
in an ESRF and the ETPF. 

More specifically, the LETKF is based on the following modifications
to the ESRF. First one replaces (\ref{sec5:eq:Stransform}) by
\[
Q(x) = \left\{ I +   \frac{1}{M-1} (A_y^f)^T \tilde R^{-1}(x)  A_y^f  \right\}^{-1}
\]
and defines $D(x) = Q(x)^{1/2}$. Finally the localized transformation
matrix $S(x)$ is given by
\begin{equation} \label{sec5:LETKF}
s_{ij}(x) = \frac{1}{M-1} q_{ij}(x) (y^f_j - \bar y^f)
\tilde R^{-1}(x) (y_{\rm obs} - \bar y^f) + d_{ij}(x),
\end{equation}
which replaces (\ref{sec5:letkf}). We mention that
\cite{sr:anderson12b} discusses practical methods for choosing 
the localization radius $r_{{\rm loc},R}$ for EnKFs.

In order to extend the concept of $R$-localization to the ETPF, we also
define the localized cost function
\begin{equation} \label{sec5:localization2}
c_x(z^f,z^a) = \int_0^L {\cal K}(x,x';r_{{\rm loc},c})
\|z^f(x')-z^a(x')\|^2 {\rm d}x'
\end{equation}
with a localization radius $r_{{\rm loc},c}\ge 0$, which can be chosen
independently from the localization radius for the measurement error
covariance matrix $R$.

The ETPF with R-localization can now be implemented as follows.
At each spatial location $x\in [0,L)$ one determines the 
desired transformation matrix  $S(x)$ by first computing the weights
\begin{equation} \label{sec5:locweights}
w_i \propto e^{-\frac{1}{2}(h(z_i^f) - y_{\rm obs})^T \tilde
  R^{-1}(x) (h(z_i^f)-y_{\rm obs}) }
\end{equation}
and then minimizing the cost function
\begin{equation} \label{sec5:locdist}
J(T) = \sum_{i,j=1}^M c_x(z_i^f,z_j^f) t_{ij}
\end{equation}
over all admissible couplings. One finally sets $S(x) = M T^\ast$.

As discussed earlier any infinite-dimensional evolution equation such as the linear
advection equation will be truncated in practice to a computational
grid $x_j = j\Delta x$. The transform matrices $S(x)$ need then to be
computed for each grid point only and the integral in
(\ref{sec5:localization2}) is replaced by a simple Riemann sum. 

We mention that alternative filtering strategies for spatio-temporal processes have been
proposed by \cite{sr:majda} in the context of turbulent systems. One of their strategies
is to perform localization in spectral space in case of regularly spaced observations. 
Another spatial localization strategy for particle filters can be found in \cite{sr:vanhandel14}. 


%
%
%
%

\section{Applications}
\label{sec:6}

In this section we present some numerical results comparing the
different LETFs for the chaotic Lorenz-63 \citep{sr:lorenz63} and
Lorenz-96 \citep{sr:lorenz96} models. While the highly nonlinear
Lorenz-63 model can be used to investigate the behavior of different
DA algorithms for strongly non-Gaussian distributions, the forty
dimensional Lorenz-96 model is a prototype ``spatially extended''
system which demonstrates the need for localization in order to achieve
skillful filter results for moderate ensemble sizes. We begin with the
Lorenz-63 model.

We mention that theoretical results on the long time behavior of filtering algorithms for chaotic
systems, such as the Lorenz-63 model, have been obtained, for example, by
\cite{sr:hunt13} and \cite{sr:law13}.

\subsection{Lorenz-63 model}

The Lorenz-63 model is given by the differential equation
(\ref{sec2:ODE}) with state variable $z = ({\rm x},{\rm y},{\rm z})^T
\in \mathbb{R}^3$, right hand side
\begin{align*}
f(z) &= \left( \begin{array}{l} \sigma ({\rm y}-{\rm x})\\
{\rm x}(\rho-{\rm z}) - {\rm y} \\ {\rm xy} - \beta {\rm z}
\end{array} \right),
\end{align*}
and parameter values $\sigma = 10$, $\rho = 28$, and $\beta = 8/3$.
The resulting ODE (\ref{sec2:ODE}) is discretized in time by the
implicit midpoint method \citep{sr:ascher08}, \emph{i.e.},
\begin{equation} \label{sec5:IM}
z^{n+1} = z^n + \Delta t f(z^{n+1/2}) ,\qquad z^{n+1/2} = \frac{1}{2}
(z^{n+1} + z^n)
\end{equation}
with step-size $\Delta t = 0.01$. Let us abbreviate the resulting map
by $\Psi_{\rm IM}$. Then the dynamical system (\ref{sec2:DS}) is
defined as
\[
\Psi = \Psi_{\rm IM}^{[12]}.
\]
In other words observations are assimilated every 12 time-steps. We
only observe the ${\rm x}$ variable with a Gaussian measurement error
of variance $R = 8$. 

We used different ensemble sizes from $10$ to $80$ as well
as different inflation factors ranging from $1.0$ to $1.12$ by
increments of $0.02$ for the
EnKF and rejuvenation parameters ranging from $0$ to $0.4$ by
increments of $0.04$ for the ETPF. Note that a rejuvenation parameter
of $h = 0.4$ corresponds to an inflation factor $\alpha = \sqrt{1+h^2} \approx 1.0770$.

The following variant of the ETPF with localized cost function has also
been implemented. We first compute the importance weights $w_i$ of a given
observation. Then each component of the state vector is updated using
only the distance in that component in the cost function $J(T)$. 
For example, the ${\rm x}_i^f$ components of the forecast 
ensemble members $z_i^f = ({\rm x}_i^f,{\rm y}_i^f,{\rm z}_i^f)^T$,
$i=1,\ldots,M$, are updated according to
\[
{\rm x}_i^a = M \sum_{i=1}^M {\rm x}_i^f t_{ij}^\ast
\]
with the coefficients $t_{ij}^\ast\ge 0$ minimizing the cost function
\[
J(T) = \sum_{i,j=1}^M t_{ij} |{\rm x}_i^f - {\rm x}_j^f|^2
\]
subject to (\ref{sec5:SMCM1c}). We use ETPF\_R0 as the shorthand form for this
method. This variant of the ETPF is of special interest from a 
computational point of view since the linear transport problem in $\mathbb{R}^3$
reduces to three simple one-dimensional problems. 

\begin{figure}
  \centering
  \subfigure{\includegraphics[clip=TRUE,scale=.3,trim =.5cm 0cm .5cm 0cm]{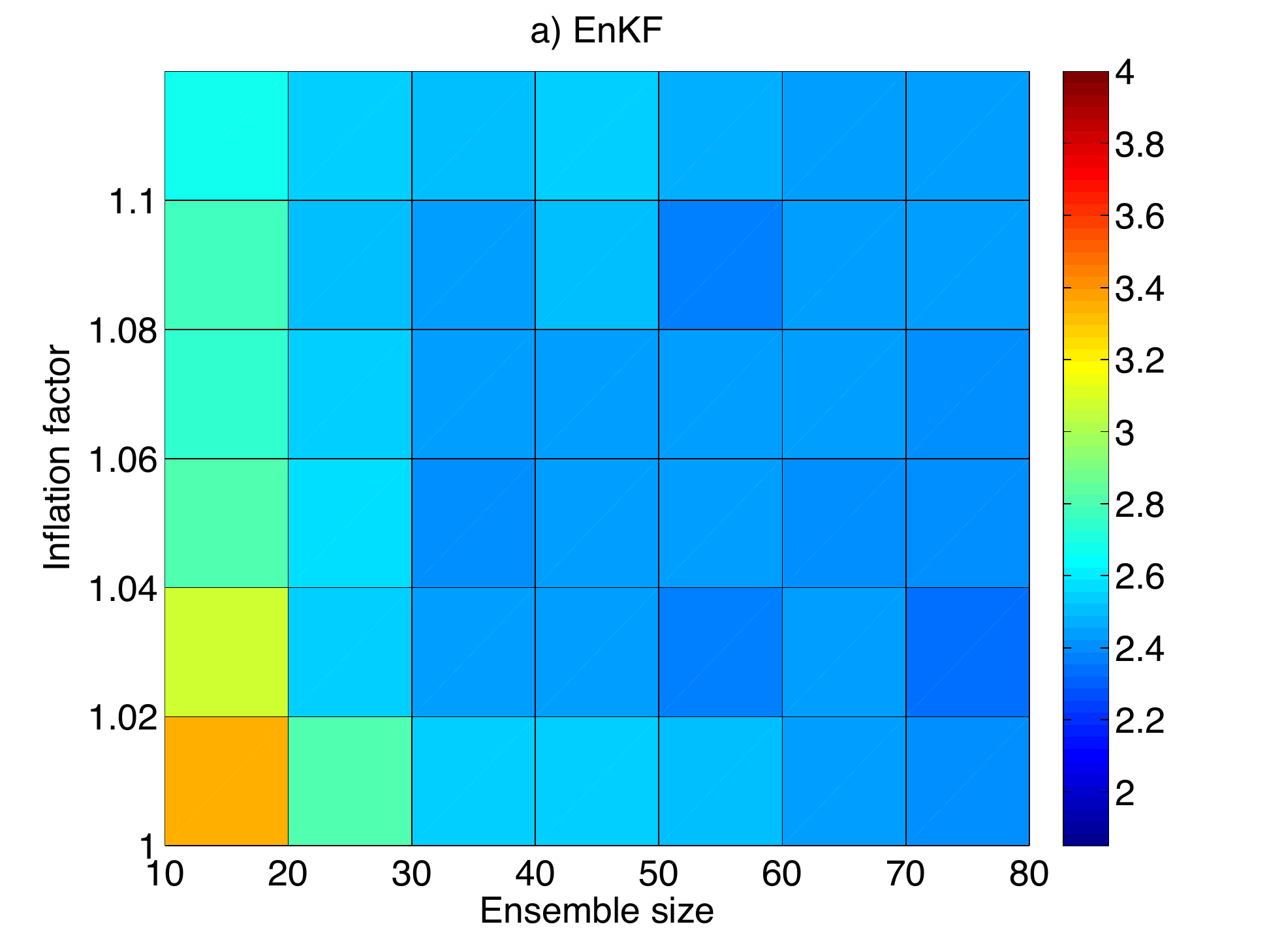}}
  \hfill
  \subfigure{\includegraphics[clip=TRUE,scale=.3,trim =.5cm 0cm .5cm 0cm]{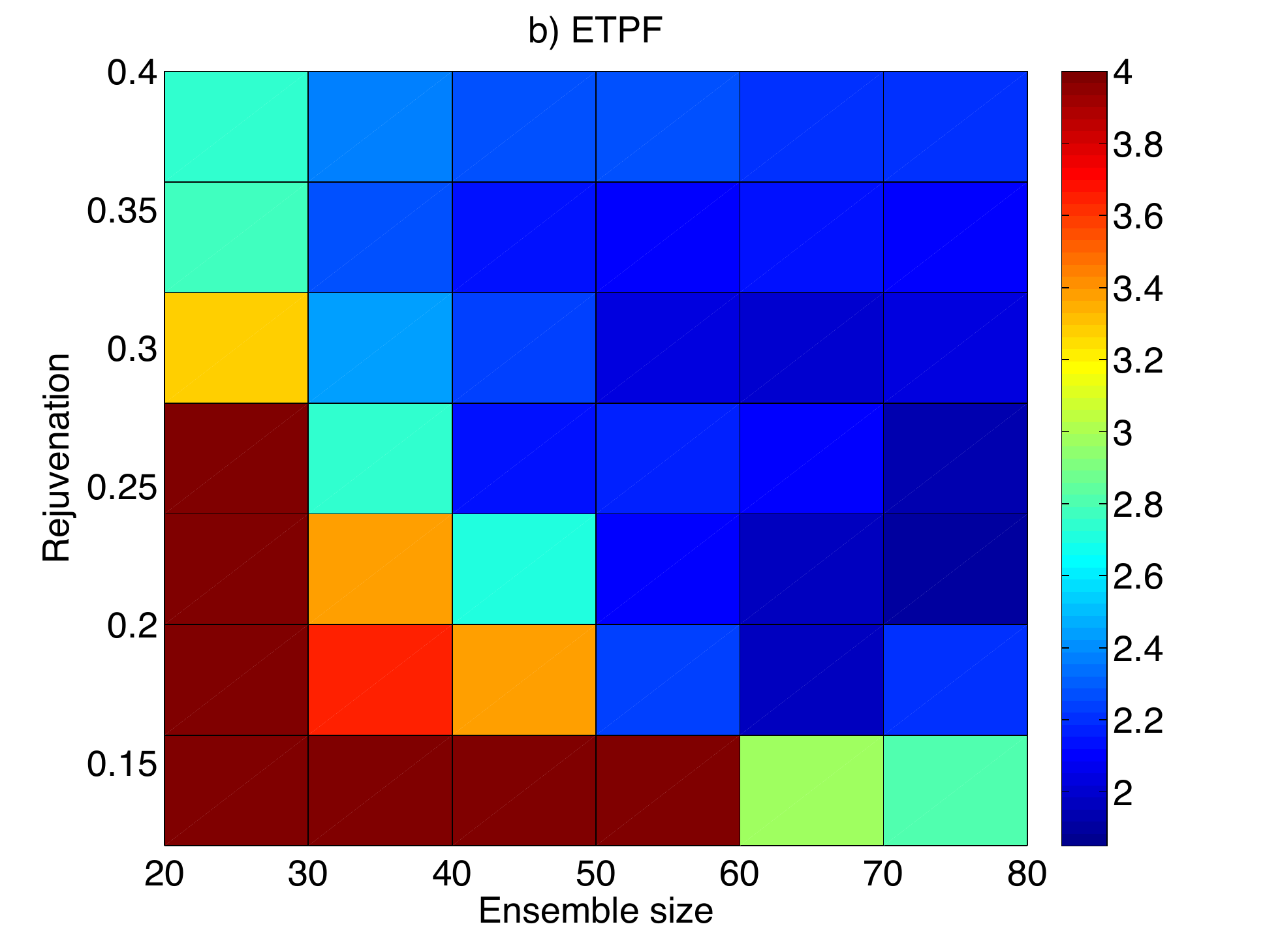}}
  \hfill
  \subfigure{\includegraphics[clip=TRUE,scale=.3,trim =0.5cm 0cm .5cm
    0cm]{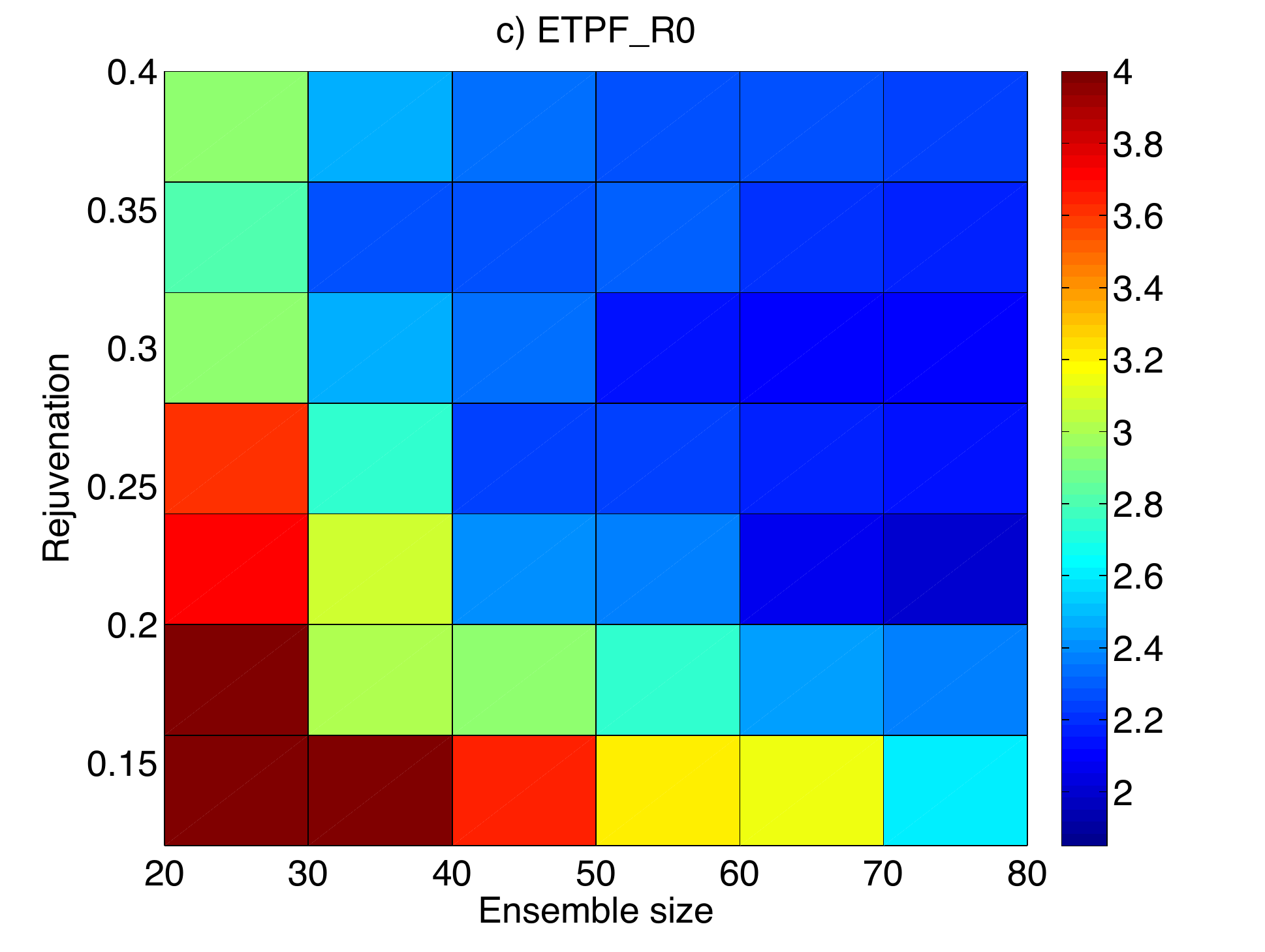}}
  \hfill
  \subfigure{\includegraphics[clip=TRUE,scale=.3,trim =0.5cm 0cm .5cm
    0cm]{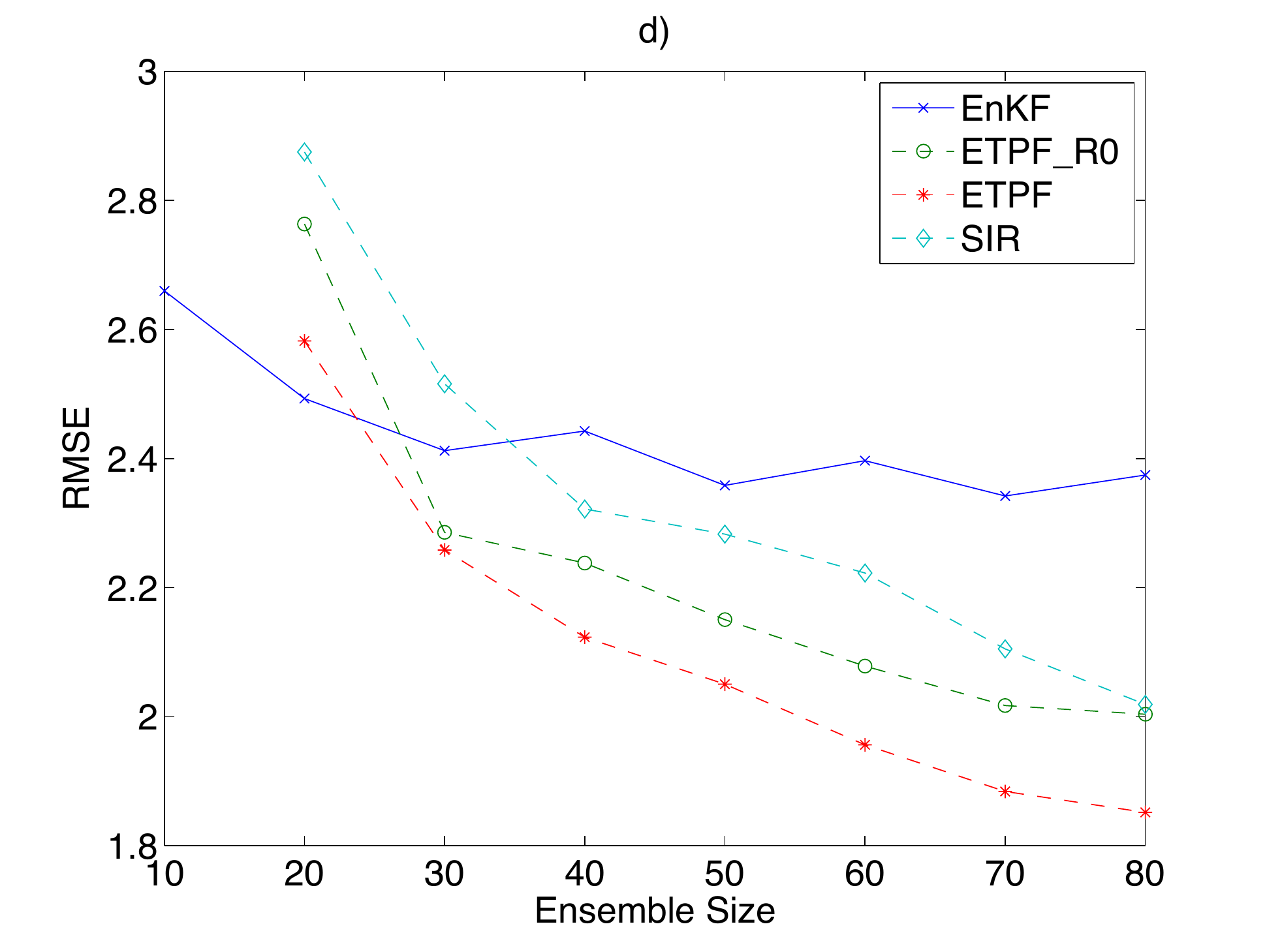}}
  \caption{a)-c): Heatmaps showing the RMSEs for different
    parameters for the EnKF, ETPF and ETPF\_R0 for the Lorenz-63
    model. d): RMSEs for different ensemble sizes using 'optimal'
    inflation factors and rejuvenation.} 
  \label{L63} 
\end{figure}

The model is run over $N=20,000$ assimilation steps after discarding
200 steps to lose the influence of the initial conditions. The
resulting \textit{root-mean-square errors} averaged over time (RMSEs)
\[
\text{RMSE} = \frac1N\sum_{n = 1}^{N}\sqrt{\|\bar{z}^{n,a}-z_{\text{ref}}^n\|^2} 
\]
 are reported in
 Fig. 2 a)-c). We dropped the results for the ETPF and ETPF\_R0 with
 ensemble 
size $M=10$ as they indicated strong divergence. We see that the EnKF
produces stable results while the other filters are more
sensitive to different choices for the rejuvenation
parameter. However, with increasing ensemble size and 'optimal' choice
of parameters the ETPF and the ETPF\_R0 outperform the EnKF which
reflects the biasedness of the EnKF. 

Fig. 2 d) shows the RMSEs for
each ensemble size using the parameters that yield the lowest
RMSE. Here we see again that the stability of the EnKF leads to good 
results even for very small ensemble sizes. The downside is also
evident: While the ETPFs fail to track the reference solution 
as well as the EnKF for very small ensemble sizes a small increase
leads to much lower RMSEs. The asymptotic consistent ETPF outperforms
the ETPF\_R0 for large ensemble sizes but is less stable otherwise. We
also included RMSEs for the SIR filter with rejuvenation 
parameters chosen from the same range of values as for the
ETPFs. Although not shown here, this range seems to cover the
'optimal' choice for the rejuvenation parameter. The comparison with
the EnKF is as expected: for small ensemble sizes the SIR performs
worse but beats the EnKF for larger ensemble sizes due to its
asymptotic consistency. However, the equally consistent ETPF yields
lower RMSEs throughout for the ensemble sizes considered
here. Interestingly, the SIR only catches up with the inconsistent but
computationally cheap ETPF\_R0 for the largest ensemble size in this
experiment. We mention that the RMSE drops to around 1.4 with the SIR
filter with an ensemble size of 1000. 

At this point we note that the computational burden increases
considerably for the ETPF for larger ensemble sizes due to the need of
solving increasingly large linear transport problems. See the discussion 
from Section \ref{secETPF}.

\subsection{Lorenz-96 model}

Given a periodic domain $x \in [0,L]$ and $N_z$ equally spaced grid-points
$x_j = j\Delta x$, $\Delta x = L/N_z$, we denote by $u_j$ the approximation to $z(x)$ at
the grid points $x_j$, $j=1,\ldots,N_z$. The following system of
differential equations
\begin{equation} \label{sec5:L96} 
\frac{{\rm d}u_j}{{\rm d} t} = -\frac{u_{j-1}u_{j+1}-u_{j-2}u_{j-1}}{3\Delta x}
- u_j + F, \qquad j = 1,\ldots,40,
\end{equation}
is due to \cite{sr:lorenz96} and is called the Lorenz-96 model. We set
$F=8$ and apply periodic boundary conditions $u_j = u_{j+40}$. The
state variable is defined by $z = (u_1,\ldots,u_{40})^T \in \mathbb{R}^{40}$.
The Lorenz-96 model (\ref{sec5:L96}) can be seen as a coarse spatial 
approximation to the PDE
\[
\frac{\partial u}{\partial t} =  -\frac{1}{2} \frac{\partial (u)^2}{
\partial x} - u + F, \qquad x \in [0,40/3],
\]
with mesh-size $\Delta x = 1/3$ and $N_z=40$ grid points. The implicit
midpoint method (\ref{sec5:IM}) is used with a step-size of $\Delta t =
0.005$ to discretize the differential equations (\ref{sec5:L96})
in time. Observations are assimilated every 22 time-steps and we
observe every other grid point with a Gaussian measurement error of
variance $R=8$. The large assimilation interval and variance of the 
measurement error are chosen because of a desired non-Gaussian ensemble 
distribution. 

We used ensemble sizes from 10 to 80, inflation factors from $1.0$ to
$1.12$ with increments of $0.02$ for the EnKF and rejuvenation
parameters between $0$ and $0.4$ with increments of $0.05$ 
for the ETPFs. 

As mentioned before, localization is required and we
take (\ref{sec5:locrho2}) as our localization kernel. For each value
of $M$ we fixed a localization radius $r_{{\rm loc},R}$ in
(\ref{sec5:localization1}). The particular choices can be
read off of the following table: 

\begin{table}
  \centering
  \begin{tabular}{c *8{ |c }}
    \hline
    M& 10&20&30&40&50&60&70&80 \\
    \hline \hline
    $r_{\text{loc},R}^{EnKF}$ &2&4&6&6&7&7&8&8 \\
    \hline
    $r_{\text{loc},R}^{ETPF}$&1&2&3&4&5&6&6&6 \\
    \hline 
  \end{tabular}
\end{table}

These values have been found by trial and error and
we do not claim that these values are by any means 'optimal'.

As for localization of the cost function
(\ref{sec5:locdist}) for the ETPF we used the same kernel as for the
measurement error and implemented different versions of the localized
ETPF which differ in the choice of the localization radius: ETPF\_R1
corresponds to the choice of $r_{\text{loc},c}= 1$ and ETPF\_R2 is
used for the ETPF with $r_{\text{loc},c}=2$. As before
we denote the computationally cheap version with cost function
$c_{x_j}(z^f,z^a) = |u_j^f - u_j^a|^2$ at grid point $x_j$ by ETPF\_R0.

\begin{figure}
  \centering
  \includegraphics[scale = .3]{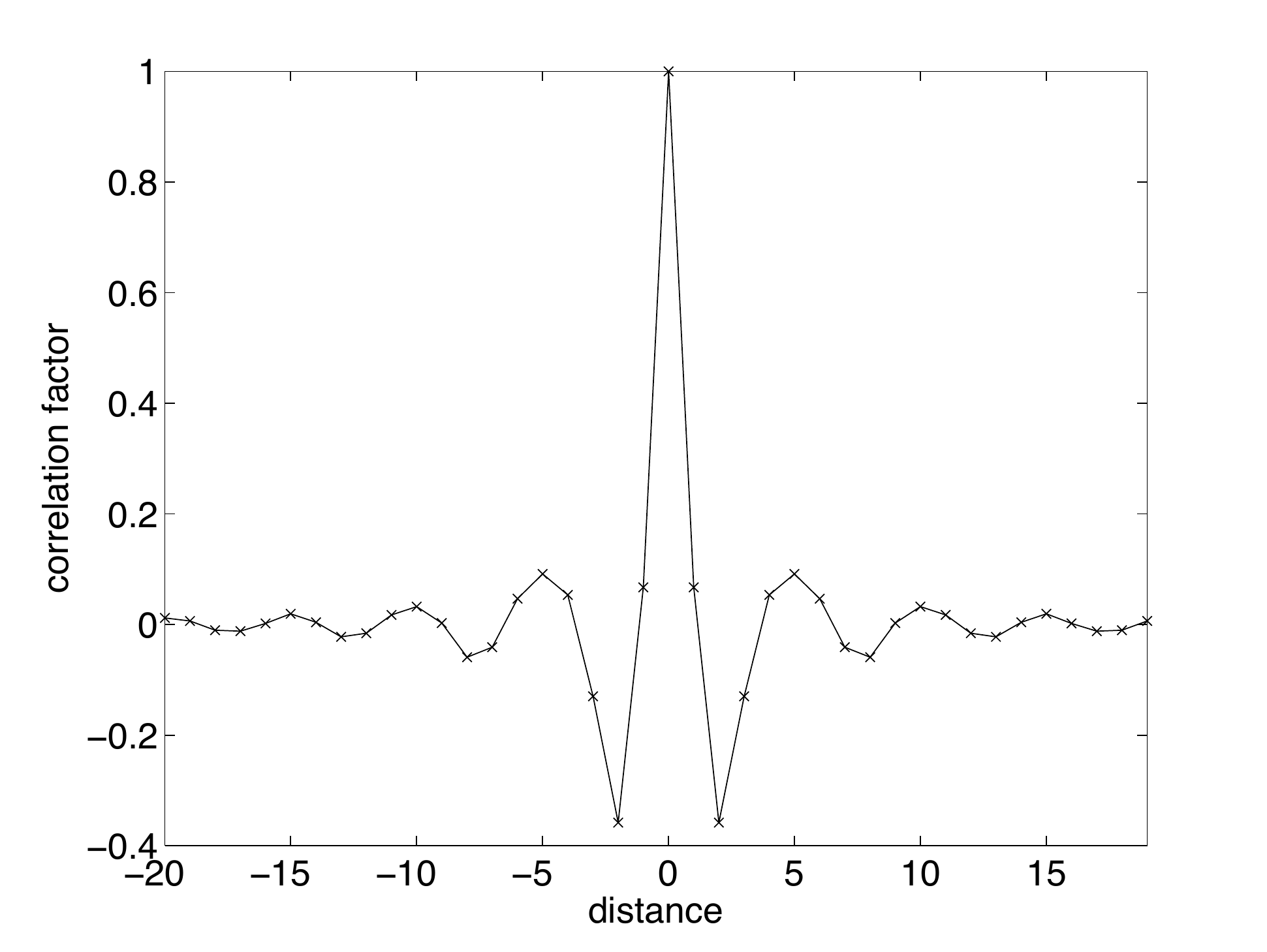}
  \caption{Time averaged spatial correlation between solution components depending on their distance.}
\end{figure}

The localization kernel as well as the localization radii
$r_{\text{loc},c}$ are not chosen by any optimality criterion but
rather by convenience and simplicity. A better kernel or localization
radii may be derived from looking at the time averaged spatial correlation
coefficients (\ref{sec5:spatialcorr}) as shown in Fig. 3. Our kernel
gives higher weights to components closer to the one to be
updated, even though the correlation with the immediate neighbor is
relatively low. 

\begin{figure}
  \centering
  \subfigure{\includegraphics[clip=TRUE,scale=.3,trim =.5cm 0cm .5cm 0cm]{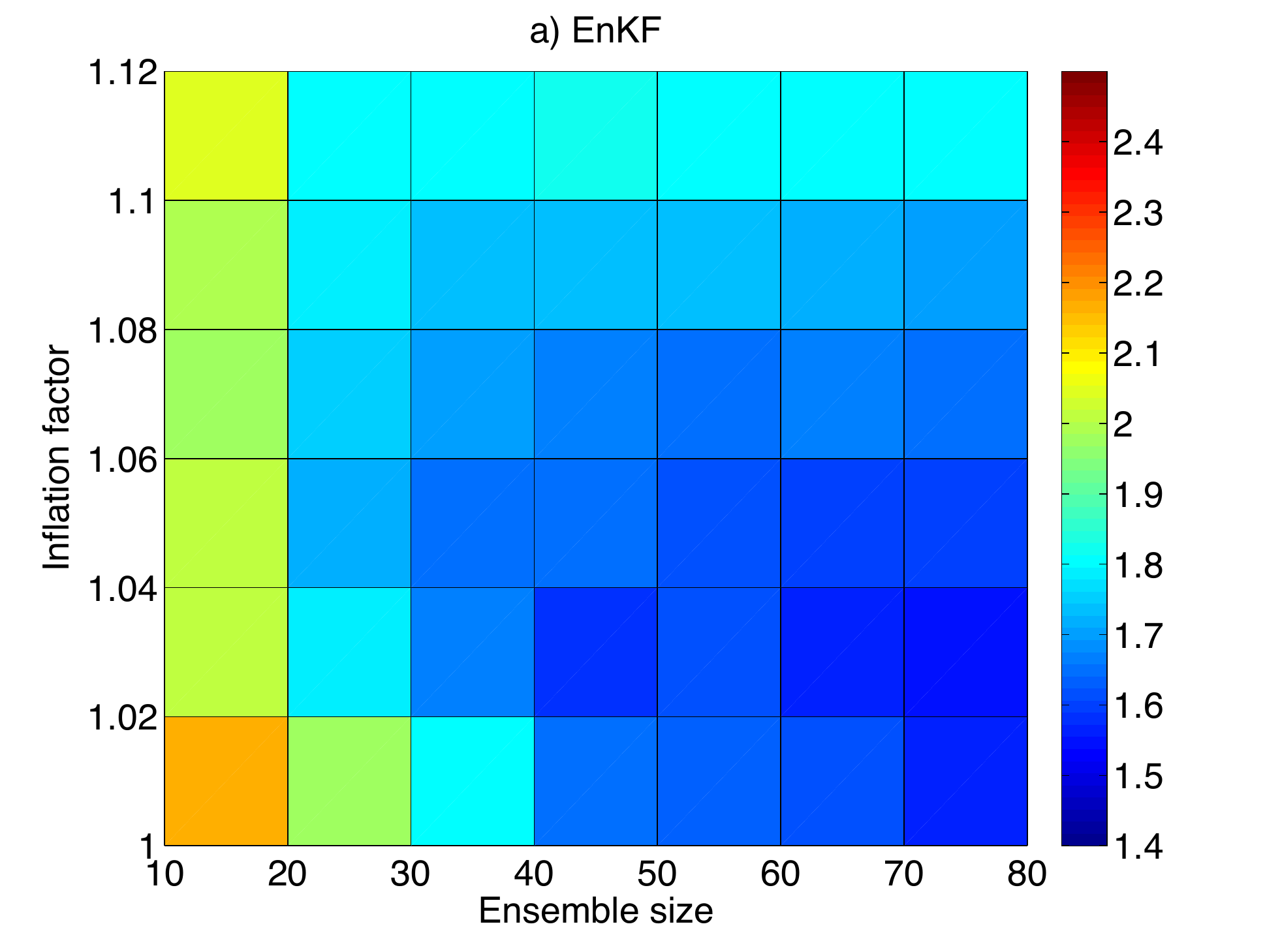}}
  \hfill
  \subfigure{\includegraphics[clip=TRUE,scale=.3,trim =.5cm 0cm .5cm 0cm]{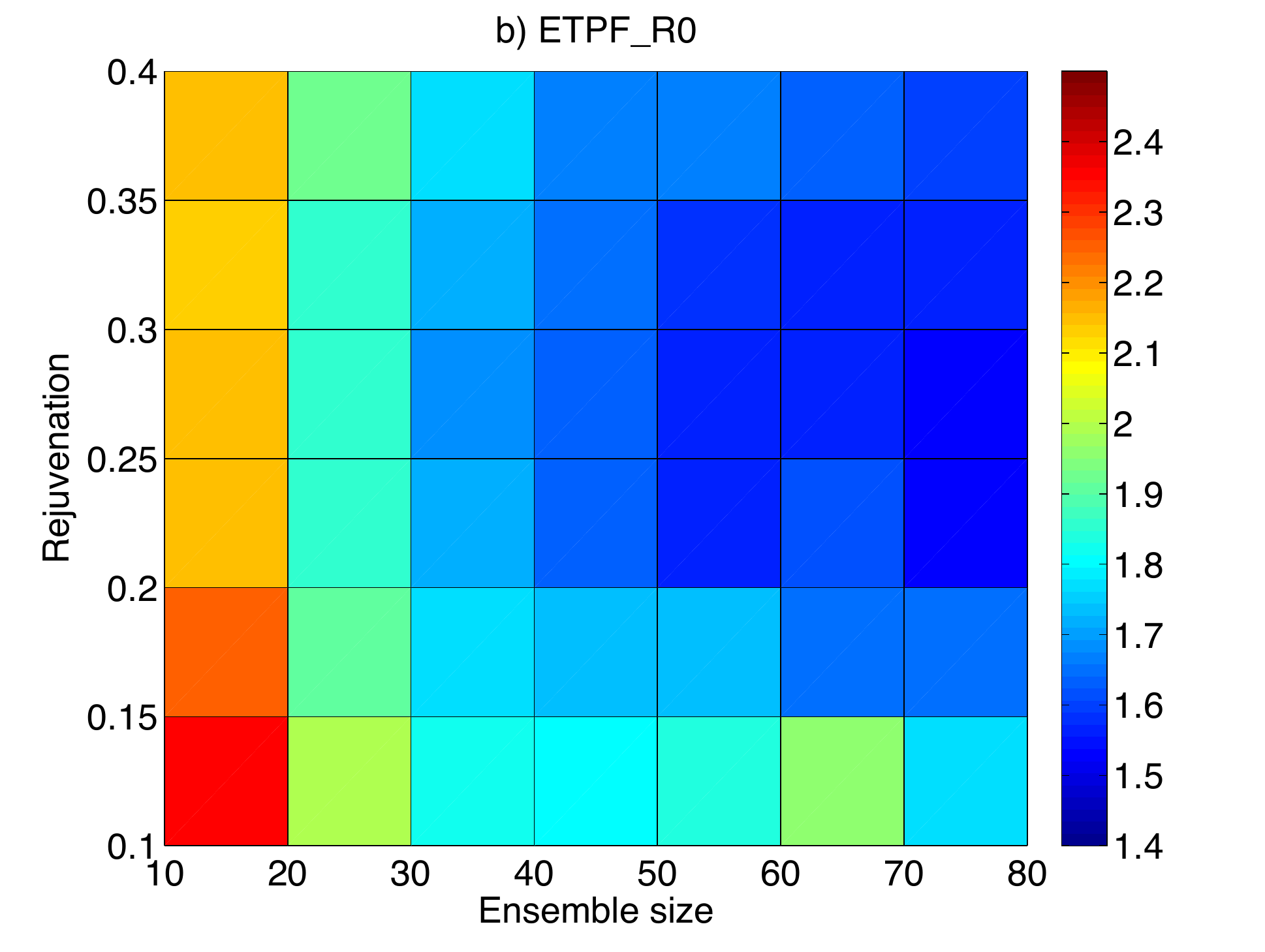}}
  \hfill
  \subfigure{\includegraphics[clip=TRUE,scale=.3,trim =0.5cm 0cm .5cm
    0cm]{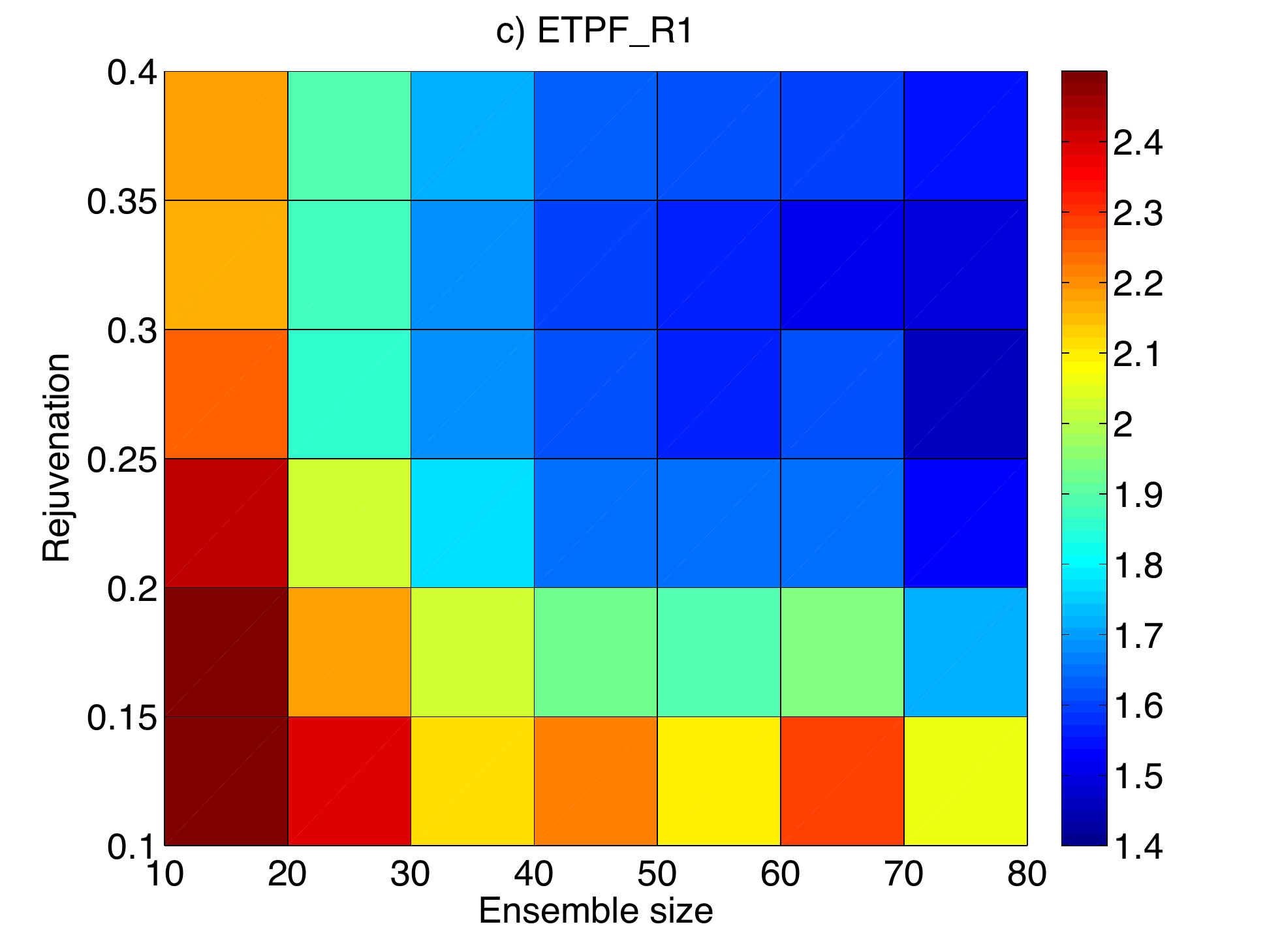}}
  \hfill
  \subfigure{\includegraphics[clip=TRUE,scale=.3,trim =0.5cm 0cm .5cm
    0cm]{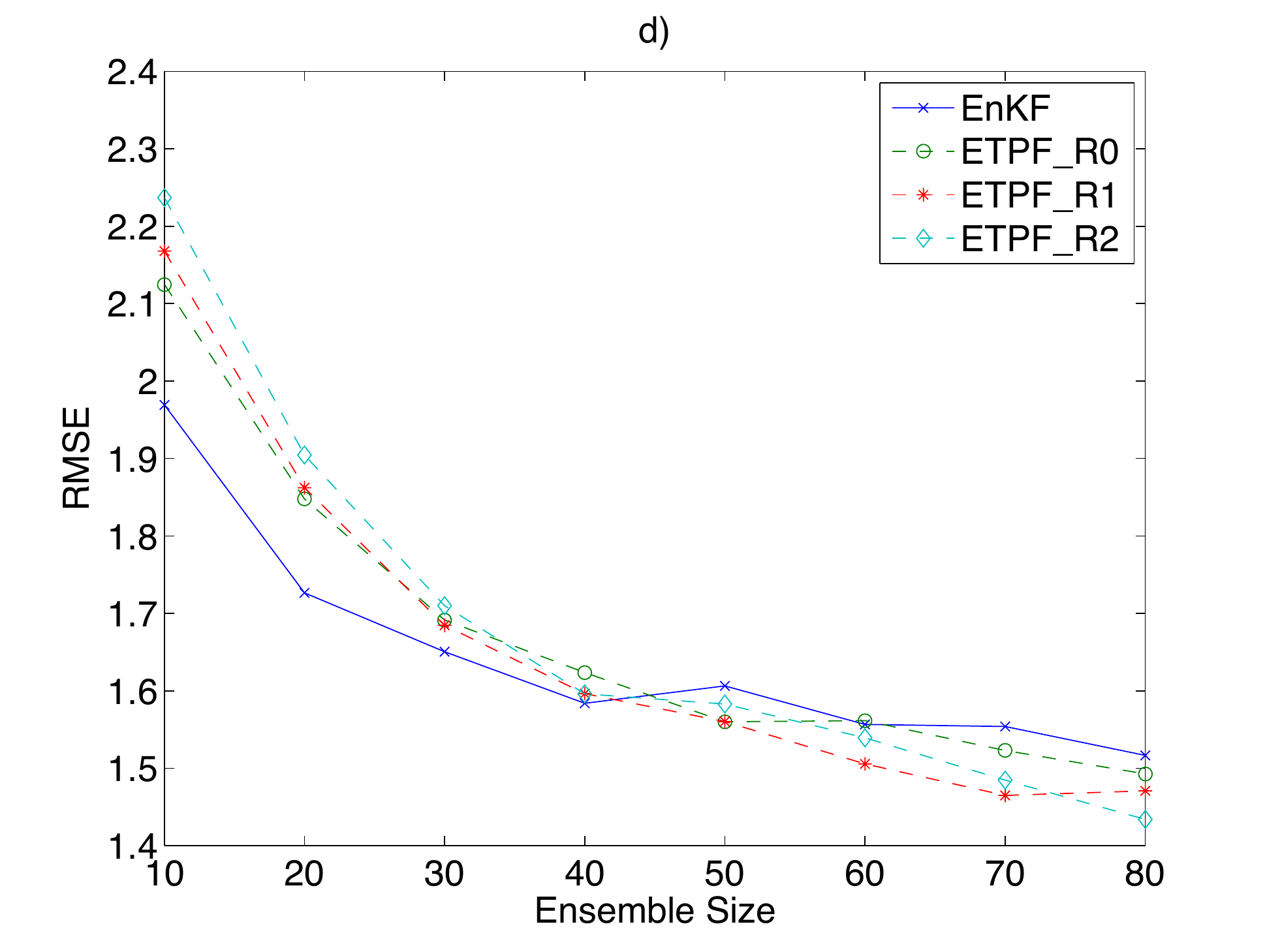}}
  \caption{a)-c): Heatmaps showing the RMSEs for different
    parameters for the EnKF, ETPF\_R0 and ETPF\_R1 for the Lorenz-96
    model. d): RMSEs for different ensemble sizes using 'optimal'
    inflation factors and rejuvenation.}
\end{figure}

The model is run over $N=10,000$ assimilation steps after discarding 500
steps to loose the influence of the initial conditions. The resulting time averaged
RMSEs are displayed in Fig. 4. We dropped the results for the
smallest rejuvenation parameters as the filters showed strong
divergence. Similar to the results for the Lorenz-63 model the EnKF
shows the most stable overall performance for various parameters but
fails to keep up with the ETPFs for higher ensemble sizes, though the
difference between the different filters is much smaller than for the
Lorenz-63 system. This is no surprise since the Lorenz-96 system does
not have the highly non-linear dynamics of the Lorenz-63 system which
causes the involved distributions to be strongly non-Gaussian. The
important point here is that the ETPF as a particle filter is able to
compete with the EnKF even for small ensemble sizes. Traditionally, high
dimensional systems required very high ensemble sizes for particle
filters to perform reasonably well. Hundreds of particles are
necessary for the SIR to be even close to the true state.

%
%
%

\section{Historical comments}

The notion of data assimilation has been coined in the field of meteorology and more widely
in the geosciences to collectively denote techniques for combining computational models 
and physical observations in order to estimate the current state of the atmosphere or any other geophysical process.  The perhaps first occurrence of the concept of data assimilation
can be found in the work of \cite{sr:Richardson}, where observational data needed to be interpolated
onto a grid in order to initialize the computational forecast process. With the rapid 
increase in computational resolution starting in the 1960s, it became quickly necessary
to replace simple data interpolation by an optimal combination of first guess estimates and 
observations. This gave rise to techniques such as the \emph{successive correction method},
\emph{nudging}, \emph{optimal interpolation} and \emph{variational least square techniques}
(3D-Var and 4D-Var). See \cite{sr:daley,sr:kalnay} for more details.  

\cite{sr:leith74} proposed \emph{ensemble} (or Monte Carlo) \emph{forecasting} as an 
alternative to conventional single forecasts. However ensemble forecasting did
not become operational before 1993 due to limited computer resources \citep{sr:kalnay}.  
The availability of ensemble forecasts subsequently lead to the invention of the 
EnKF by \cite{sr:evensen94} with a later correction by \cite{sr:burgers98} and many
subsequent developments, which have been summarized in \cite{sr:evensen}.
We mention that the analysis step of an EnKF with perturbed observations is closely related to 
a method now called \emph{randomized likelihood method} \citep{sr:kitanidis95,sr:oliver96,sr:oliver96b}.

In a completely independent line of research the problem of
optimal estimation of stochastic processes from data has led to the theory of \emph{filtering}
and \emph{smoothing},  which started with the work of \cite{sr:wiener}. The state space approach
to filtering of linear systems gave rise to the celebrated \emph{Kalman filter} and 
more generally to the \emph{stochastic PDE formulations} of Zakai and Kushner-Stratonovitch
in case of continuous-time filtering. See \cite{sr:jazwinski} for the theoretical developments
up to 1970. Monte Carlo techniques were first introduced to the filtering problem by
\cite{sr:handschin69}, but it was not until the work of \cite{sr:gordon93} that the 
SMCM became widely used \citep{sr:Doucet}.  The \emph{McKean interacting particle approach} 
to SMCMs has been pioneered by \cite{sr:DelMoral}. The theory of particle filters for time-continuous filtering problems is summarized in \cite{sr:crisan}. 

Standard SMCMs suffer from the \emph{curse of dimensionality}
in that the necessary number of ensemble members $M$ increases exponentially with
the dimension $N_z$ of state space \citep{sr:bengtsson08}. This limitation has prevented 
SMCMs from being used in meteorology and the geosciences. On the other hand, it is known 
that EnKFs lead to inconsistent estimates which is problematic when multimodal forecast distributions are 
to be expected. Current research work is therefore focused on a theoretical understanding
of EnKFs and related sequential assimilation techniques (see, for example, 
\cite{sr:hunt13,sr:law13}), extensions of particle filters/SMCMs to PDE models 
(see, for example, \cite{sr:chorin12b,sr:leeuwen12,sr:beskov13,sr:snyder13}),
and Bayesian inference on function spaces (see, for example, \cite{sr:stuart10a,sr:CDRS09,sr:dashti13})
and hybrid variational methods such as by, for example, \cite{sr:bonavita12,sr:clayton13}.

A historical account of optimal transportation can be found in
\cite{sr:Villani2}. The work of \cite{sr:mccann95} provides the
theoretical link between the classic linear assignment problem and the
Monge-Kantorovitch problem of coupling PDFs. The ETPF is a
computational procedure for approximating such couplings using importance sampling and
linear transport instead.

%
%
%

\section{Summary and Outlook}

We have discussed various ensemble/particle-based algorithms for
sequential data assimilation in the context of LETFs. Our starting
point was the McKean interpretation of Feynman-Kac formulae. The
McKean approach requires a coupling of measures which can be discussed
in the context of optimal transportation. This approach leads
to the ETPF when applied in the context of SMCMs. We have furthermore
discussed extensions of LETFs to spatially extended systems in form of 
$R$-localization. 

The presented work can be continued along several lines. First, one
may replace the empirical forecast measure
\begin{equation} \label{sec8:emp}
\pi_{\rm emp}^f(z) := \frac{1}{M} \sum_{i=1}^M \delta(z-z_i^f),
\end{equation}
which forms the basis of SMCMs and the ETPF, by a \emph{Gaussian mixture}
\begin{equation} \label{sec8:GM}
\pi_{\rm GM}^f(z) := \frac{1}{M} \sum_{i=1}^M {\rm n}(z;z_i^f,B),
\end{equation}
where $B\in \mathbb{R}^{N_z\times N_z}$ is a given covariance matrix 
and
\[
{\rm n}(z;m,B) := \frac{1}{(2\pi)^{N_z/2} |B|^{1/2}} e^{ -
  \frac{1}{2} (z-m)^T B^{-1} (z-m)} . 
\]
Note that the empirical measure (\ref{sec8:emp}) is recovered in the limit
$B\to 0$.  While the weighted empirical measure
\[
\pi_{\rm emp}^a(z) :=  \sum_{i=1}^M w_i \delta(z-z_i^f)
\]
with weights given by (\ref{sec5:SMCM1b}) provides the analysis 
in case of an empirical forecast measure (\ref{sec8:emp}) and an 
observation $y_{\rm obs}$, a Gaussian mixture
forecast PDF (\ref{sec8:GM}) leads to an analysis PDF in form of a weighted Gaussian mixture
provided the forward operator $h(z)$ is linear in $z$. This fact allows one
to extend the ETPF to Gaussian mixtures. See \cite{sr:reichcotter14} for
more details. Alternative implementations of Gaussian mixture filters can, for example, 
be found in \cite{sr:stordal11} and \cite{sr:frei11}.

Second, one may factorize the likelihood function
$\pi_{Y^{1:N}}(y^{1:N}|z^{0,N})$ into $L> 1$ identical copies
\[
\hat \pi_{Y^{1:N}}(y^{1:N}|z^{0:N}) := \frac{1}{(2\pi)^{N_yN/2}
  |R/L|^{N/2}} \prod_{n=1}^N e^{- \frac{1}{2L} (h(z^n)-y^n)^T R^{-1} (h(z^n)-y^n) },
\]
\emph{i.e.},
\[
\pi_{Y^{1:N}}(y^{1:N}|z^{0,N}) = \prod_{l=1}^L
\hat \pi_{Y^{1:N}}(y^{1:N}|z^{0:N})
\]
and one obtains a sequence of $L$ ``incremental'' Feynman-Kac
formulae. Each of these formulae can be approximated numerically by any
of the methods discussed in this review. For example, one
obtains the continuous EnKF formulation of \cite{sr:br10} in the
limit $L\to \infty$  in case of an ESRF. We also mention the continuous Gaussian mixture 
ensemble transform filter \citep{sr:reich11}. An important advantage
of an incremental approach is the fact that the associated weights (\ref{sec5:SMCM1b})
remain closer to the uniform reference value $1/M$ in each iteration step. 
See also related methods such as \emph{running in place}
(RIP) \citep{sr:kalnayyang10}, the iterative EnKF approach of 
\cite{sr:bocquet12,sr:sakov12}, and the embedding approach of
\cite{sr:beskov13} for SMCMs.

Third, while this paper has been focused on discrete time
algorithms, most of the presented results can be extended to
differential equations with observations arriving continuously in time
such as
\[
{\rm d}y_{\rm obs}(t) = h(z_{\rm ref}(t)){\rm d} t + \sigma {\rm d}W(t),
\]
where $W(t)$ denotes standard Brownian motion and $\sigma> 0$
determines the amplitude of the measurement error. The associated marginal
densities satisfy the Kushner-Stratonovitch stochastic PDE
\citep{sr:jazwinski}. Extensions of the McKean approach to
continuous-in-time filtering problems can be found in \cite{sr:crisan10} and \cite{sr:meyn13}.
We also mention the continuous-in-time formulation of the EnKF by
\cite{sr:br11}. More generally, a reformulation of LETFs in terms of
continuously-in-time arriving observations is of the abstract
form
\begin{equation} \label{sec6:ip}
{\rm d}z_j    = f(z_j){\rm d}t + \sum_{i=1}^M z_i {\rm d}s_{ij},
\qquad j = 1,\ldots,M .
\end{equation}
Here $S(t) = \{s_{ij}(t)\}$ denotes a matrix-valued stochastic process which
depends on the ensemble $\{z_i(t)\}$ and the observations $y_{\rm
  obs}(t)$. In other words, (\ref{sec6:ip}) leads to a particular class of
interacting particle systems and we leave further investigations of
its properties for future research. We only mention that
the continuous-in-time EnKF formulation of \cite{sr:br11}
leads to
\[
{\rm d}s_{ij} = \frac{1}{M-1} (y_i-\bar y)\sigma^{-1}({\rm d}y_{\rm obs} -
y_j{\rm d}t + \sigma^{1/2} {\rm d}W_j ),
\] 
where the $W_j(t)$'s denote standard Brownian motion, $y_j = h(z_j)$,
and $\bar y = \frac{1}{M}\sum_{i=1}^M y_i$. See als
\cite{sr:akir11} for related reformulations of ESRFs.

\begin{acknowledgement}
We would like to thank Yann Brenier, Dan Crisan, Mike Cullen and Andrew Stuart
for inspiring discussions on ensemble-based filtering methods and the theory of 
optimal transportation. 
\end{acknowledgement}
%


\bibliographystyle{plainnat}
\bibliography{survey}

\end{document}